%% file: main.tex
\documentclass[11pt]{amsart}

\usepackage[margin=1in]{geometry}
\usepackage{amsmath, amssymb, amsthm, mathtools}
\usepackage{hyperref}
\usepackage{placeins}
\usepackage[colorinlistoftodos]{todonotes}
\hypersetup{
    colorlinks=true,
    linkcolor=blue,
    citecolor=blue,
    urlcolor=blue
}
\usepackage{tikz}
\usetikzlibrary{patterns.meta,arrows.meta}

\newtheorem{theorem}{Theorem}[section]
\newtheorem{lemma}[theorem]{Lemma}
\newtheorem{proposition}[theorem]{Proposition}
\newtheorem{corollary}[theorem]{Corollary}

\theoremstyle{definition}

\theoremstyle{remark}
\newtheorem{remark}[theorem]{Remark}

\newcommand{\R}{\mathbb{R}}

\newcommand{\cP}{\mathcal{P}}
\newcommand{\cA}{\mathcal{A}}

\newcommand{\E}{\mathbb{E}}

\newcommand{\1}{\mathbf{1}}

\newcommand{\diam}{\operatorname{diam}}

\newcommand{\Lip}{\operatorname{Lip}}

\newcommand{\Sd}{\mathbb{S}}

\newcommand{\one}{\mathbf{1}}
\DeclarePairedDelimiter{\abs}{\lvert}{\rvert}

\numberwithin{equation}{section}

\newcommand{\TV}{\mathrm{TV}}

\title{Sample Complexity for the 2-Gromov--Wasserstein Distance}

\author{Pui Kuen Leung}
\address[Pui Kuen Leung]{Department of Statistics, University of Chicago}
\email{\texttt{pkl@uchicago.edu}}
\author{Riku Okada}
\address[Riku Okada]{Committee on Computational and Applied Mathematics, University of Chicago}
\email{\texttt{rikuo@uchicago.edu}}
\author{Samuel Lok-Hei Wong}
\address[Samuel Lok-Hei Wong]{Committee on Computational and Applied Mathematics, University of Chicago}
\email{\texttt{slhwong@uchicago.edu}}
\thanks{The authors are listed in alphabetical order.}
\date{\today}
\begin{document}

\begin{abstract}
In this paper, we study the sample complexity of the empirical plug-in estimator for the
$2$-Gromov--Wasserstein distance $D_2$ between compactly supported probability measures on 
Euclidean spaces. Let $\mu$ and $\nu$ be supported on compact subsets of 
$\mathbb{R}^{d_x}$ and $\mathbb{R}^{d_y}$, respectively, and let $\widehat\mu_n$ and $\widehat\nu_n$ be their empirical measures based on independent samples of size $n$. We prove that
\[
\mathbb{E}\left|D_2^2(\widehat\mu_n,\widehat\nu_n)-D_2^2(\mu,\nu)\right|
\lesssim
n^{-2/((d_x\wedge d_y)\vee 4)}
(\log n)^{\mathbf 1_{\{d_x\wedge d_y=4\}}}.
\]
This rate is sharp up to the logarithmic factor in the critical dimension. The proof is based on a geometric representation of the Euclidean distance as a squared $L^2$-distance between half-space feature maps. This yields a variational dual formulation of the Gromov--Wasserstein functional in terms of a family of classical optimal transport problems indexed by an infinite-dimensional auxiliary parameter. Although the resulting cost functions need not be semiconcave in either argument, we introduce a marginal recentering of the costs that restores the concavity structure needed for sharp metric-entropy bounds. Combining this representation with empirical-process estimates gives a rate governed by the smaller of the two ambient dimensions.
\end{abstract}
\maketitle 



\input{sections/introduction}
\input{sections/preliminaries}
\input{sections/crofton-2}

\input{sections/dual}
\input{sections/proofs-main-results}
\bibliographystyle{plain}
\bibliography{references}
\clearpage
\appendix
\input{appendix/proofs-crofton-duality}
\input{appendix/proofs-crofton-halfspace}
\input{appendix/dual}
\input{appendix/proofs-duality-cost}
\input{appendix/proofs-sample-complexity}
\input{appendix/semiconcave-failure}

\end{document}

%% file: sections/introduction.tex
\section{Introduction}
The Gromov--Wasserstein (GW) distance, first introduced by M\'emoli~\cite{memoli2011gromov}, provides a framework for comparing probability measures supported on different metric spaces. For metric measure
spaces $(\mathcal X,d_X,\mu)$ and $(\mathcal Y,d_Y,\nu)$, the $p$-Gromov--Wasserstein distance is defined as
\begin{equation}
\label{eq:gw-distance}
    D_{p}(\mu,\nu)
    :=
    \inf_{\gamma \in \Pi(\mu,\nu)}
    \left(
        \iint_{(\mathcal X\times \mathcal Y)^2}
        \left|
            d_X(x,x') - d_Y(y,y')
        \right|^p
        \,d\gamma(x,y)\,d\gamma(x',y')
    \right)^{1/p},
\end{equation}
where the infimum is taken over all couplings $\gamma\in \Pi(\mu,\nu)$ with marginals $\mu$ and $\nu$. Unlike classical optimal transport (OT), the cost is not prescribed directly between a point $x \in \mathcal X$ and a point $y \in \mathcal Y$. Instead, a coupling $\gamma$ is interpreted as a correspondence between the two metric spaces, and the objective penalizes the distortion of pairwise distances under this correspondence. This makes the GW distance a natural object of study in many geometric data analysis applications when comparing shapes, graphs, and other relational data that lack a common coordinate system or canonical alignment. Since its introduction, the GW distance has become a widely used tool in computational optimal transport and geometric data analysis~\cite{chowdhury2019gromov,peyre2016gromov,vayer2019optimal,xu2019gromov}. 

This naturally raises the need to understand its statistical properties. In applications, one often observes finitely many samples from a probability distribution rather than the distribution itself. Thus, an important statistical problem is to understand the convergence of the empirical plug-in estimator $D_{p}(\widehat\mu_n,\widehat\nu_n)$ to the true distance $D_{p}(\mu,\nu)$. While it is known that this convergence holds almost surely~\cite{memoli2011gromov}, the sharp nonasymptotic rate of convergence had remained unknown, even in expectation and in the Euclidean setting. In contrast, for classical OT, this problem is by now well studied. Indeed, the study of empirical OT convergence dates at least to Dudley~\cite{dudley1969speed}, who used the Kantorovich dual formulation for the $1$-Wasserstein distance $W_1$ together with metric-entropy methods to obtain convergence rates for empirical measures on totally bounded metric spaces. In particular, for a compactly supported probability measure $\mu$ on $\R^d$, Dudley's result yields
$\E[W_1(\widehat\mu_n,\mu)]\lesssim n^{-1/d}$
when $d>2$. Later work~\cite{dereich2013constructive} proves, for a probability measure $\mu$ on $\R^d$ satisfying suitable moment conditions, the bound
$\E[W_p^p(\widehat\mu_n,\mu)]\lesssim n^{-p/d}$
when $1\leq p<d/2$. ~\cite{fournier2015rate} extends this analysis to every $p>0$, covering also the critical and low-dimensional regimes $d=2p$ and $d<2p$. Further work~\cite{weed2019sharp} develops sharp asymptotic and finite-sample bounds on more general compact metric spaces, while ~\cite{lei2020convergence} extends related upper-bound and concentration techniques to measures with unbounded support on infinite-dimensional spaces. Beyond Wasserstein-distance estimation, ~\cite{manole2024sharp} establishes sharp nonasymptotic rates for plug-in estimators of OT costs induced by a broader class of cost functions satisfying suitable smoothness and structural assumptions. For classical OT, the linear dependence of the objective on the coupling $\gamma$ makes the analysis tractable. By contrast, the GW objective depends quadratically on the coupling through $\gamma\otimes\gamma$ and is generally nonconvex in $\gamma$, making its analysis substantially more difficult.

The work of Zhang et al.~\cite{zhang2024gromov} initiated the study of empirical GW estimation by considering the sample complexity of the $(2,2)$-GW distance for compactly supported distributions on the Euclidean spaces $\R^{d_x}$ and $\R^{d_y}$, where the $(p,q)$-GW distance is defined by
\[ D_{p,q}(\mu,\nu)
    :=
    \inf_{\gamma \in \Pi(\mu,\nu)}
    \left(
        \iint
        \left|
            \|x-x'\|^q - \|y-y'\|^q
        \right|^p
        \,d\gamma(x,y)\,d\gamma(x',y')
    \right)^{1/p}.\]
Thus, whereas the usual $D_p$ objective compares pairwise distances, the $(2,2)$-GW distance compares squared pairwise distances.  This is less natural from a metric geometry perspective, but is mathematically convenient since the Euclidean identity
\[\|x - x'\|^2 = \|x\|^2 + \|x'\|^2 - 2\langle x,x'\rangle\] 
allows the objective to be expanded to identify the quadratic dependence of $\gamma$ as the matrix $\int xy^\top\,d\gamma(x,y)$. This quadratic term can then be linearized through a finite-dimensional variational representation, reducing the problem to a family of standard optimal transport problems whose objectives depend linearly on $\gamma$. This makes duality and empirical-process techniques available. Using the variational representation, Zhang et al.~\cite[Theorem~4.2]{zhang2024gromov} show that the plug-in error
$\E\bigl|D_{2,2}^2(\widehat\mu_n,\widehat\nu_n)
-D_{2,2}^2(\mu,\nu)\bigr|$
is bounded at the rate
$n^{-2/((d_x\wedge d_y)\vee4)}
(\log n)^{\mathbf 1_{\{d_x\wedge d_y=4\}}}$,
which is sharp up to the logarithmic factor in the critical case $d_x\wedge d_y=4$ and has the same form as the empirical $W_2^2$ rate for compactly supported measures~\cite{fournier2015rate}, with the dimension $d$ replaced by the smaller ambient dimension $d_x\wedge d_y$. Later work by Kato and Wang~\cite{kato2025convergence} extends this sample complexity to unbounded marginals under finite polynomial-moment assumptions.

This linearization approach is not restricted to the $(2,2)$ case. When $p=2r$ and $q=2k$ are even integers, the powered functional $D_{2r,2k}^{2r}$ admits a finite polynomial expansion, and recent work by Paliy~\cite{paliy2026empirical} extends the variational reduction and sample-complexity analysis to this even-order family. Nevertheless, these arguments rely on taking even powers of Euclidean distances and do not extend directly to the usual $D_p=D_{p,1}$ distance. Consequently, although general comparisons with Wasserstein distances directly yield the naive rate $n^{-1/(d_x\vee d_y)}$ (cf. ~\cite[Remark~4.5]{zhang2024gromov}), sharper convergence rates for the empirical estimator of $D_p$ remain open.

\subsection{Main result and proof overview} 
In this paper, we address this gap for the $2$-GW distance. We prove that its empirical plug-in estimator achieves the same
sharp sample-complexity rate as that established for $(2,2)$-GW in
\cite{zhang2024gromov}. Thus, the rate obtained for $(2,2)$-GW
is not merely a consequence of replacing pairwise distances by their squares.
The formal statement is as follows.

\begin{theorem}[Main result]
\label{thm:main-informal}
Let $\mathcal X\subset\R^{d_x}$ and $\mathcal Y\subset\R^{d_y}$ be
compact, and let $\mathcal P(\mathcal X)$ and $\mathcal P(\mathcal Y)$ denote the spaces of Borel probability measures on $\mathcal X$ and $\mathcal Y$, respectively. Then there exists a constant $C$, depending only on
$\diam\mathcal X\vee\diam\mathcal Y$, $d_x$, and $d_y$, such that, for
every $n\geq2$, every $\mu\in\mathcal P(\mathcal X)$, and every
$\nu\in\mathcal P(\mathcal Y)$,
\begin{equation}
\label{eq:upper-bound}
\E\left|
D_2^2(\widehat\mu_n,\widehat\nu_n)-D_2^2(\mu,\nu)
\right|
\leq
C n^{-\frac{2}{(d_x\wedge d_y)\vee4}}
(\log n)^{\mathbf 1_{\{d_x\wedge d_y=4\}}},
\end{equation}
where $\widehat\mu_n$ and $\widehat\nu_n$ are the empirical measures
associated with i.i.d.\ samples from $\mu$ and
$\nu$, independent of each other.
Furthermore, this rate is sharp up to the logarithmic factor in the sense that there exists a constant $c>0$, depending only on
$d_x$ and $d_y$, such that, for all sufficiently large $n$,
\begin{equation}
\label{eq:lower-bound}
\sup_{\substack{
\mu\in\mathcal P(B_1^{d_x}(0)),
\nu\in\mathcal P(B_1^{d_y}(0))
}}
\E\left|
D_2^2(\widehat\mu_n,\widehat\nu_n)-D_2^2(\mu,\nu)
\right|
\geq
c n^{-\frac{2}{(d_x\wedge d_y)\vee4}},
\end{equation}
where
$B_1^d(0)$ is the closed unit ball in $\R^d$.
\end{theorem}

The key tool in our analysis is a duality theory for $D_{2}^2$ based on
the Crofton formula from integral geometry, which allows us to represent the
Euclidean distance between two points in terms of whether a random affine half-space separates them ~\cite{Santalo2004}. Associating to each point $x$ an $L^2$ feature map $U_x$ recording whether it lies in a given affine half-space, we
write the Euclidean distance as the squared $L^2$ norm between the associated
feature maps:
\[\|x-x'\| = \|U_x - U_{x'}\|^2_{L^2}.\]
This enables the expansion approach in the same spirit as the analysis of $(2,2)$-GW in \cite{zhang2024gromov}, leading to a dual formulation of the form
\begin{equation}\label{eq:duality-intro}
D_{2}^2(\mu,\nu)
= E_{\mu,\nu}
+ \inf_{a \in \mathcal B_{\mu,\nu}}
\Bigl\{ \tfrac{1}{8}\|a\|_{L^2}^2 + T_{c_{a,\mu,\nu}}(\mu,\nu) \Bigr\},
\end{equation}
where $E_{\mu,\nu}$ collects the marginal-dependent terms,
$\mathcal B_{\mu,\nu}$ is a norm-bounded subset of an infinite-dimensional $L^2$
space, and $T_{c_{a,\mu,\nu}}$ is an ordinary OT problem with cost
$c_{a,\mu,\nu}$ parameterized by the auxiliary parameter $a$
(Theorem~\ref{thm:global-duality}). This formulation requires no 
compactness or centering assumption on the marginals, and in particular applies directly
to the empirical measures $\widehat\mu_n$ and $\widehat\nu_n$. Although the metric entropy of the parameter class $\mathcal B_{\mu,\nu}$ is not tractable due to the infinite-dimensionality of the ambient $L^2$ space, under the additional compactness assumption on the marginals, the optimization in \eqref{eq:duality-intro} can instead be restricted to a more structured parameter class $\mathcal A$ whose metric entropy can be controlled (Corollary \ref{cor:compact-duality}). This makes the dual formulation suitable for the sample-complexity analysis.

With the above dual formulation in hand, our sample-complexity analysis
follows standard empirical process arguments based on metric-entropy bounds for classes of
Kantorovich potentials in classical OT and
$(2,2)$-GW ~\cite{manole2024sharp,zhang2024gromov}. However, unlike the
cost functions arising for $W_2^2$ and $(2,2)$-GW and more generally $(2r,2k)$-GW~\cite{paliy2026empirical}, the induced costs
$c_{a,\mu,\nu}$ in our dual formulation need not be semiconcave in either argument; without further structure, one would only obtain the weaker metric-entropy bounds for Lipschitz classes and hence the slower $W_1$-type rate $n^{-1/(d_x\wedge d_y)}$ when $d_x\wedge d_y>2$. We overcome this issue by recentering the costs by a correction term depending only on one marginal (see the proof outline below), with which we recover the concavity structure needed for the sharp metric-entropy bounds underlying the target rate. The full proof of Theorem \ref{thm:main-informal} is given in Section \ref{sec:proofs-main-results}, and below we give an outline.

Without loss of generality, assume
$d_x\leq d_y$. For the upper bound~\eqref{eq:upper-bound},
applying the dual formulation ~\eqref{eq:21dual} to both $(\mu,\nu)$ and
$(\widehat\mu_n,\widehat\nu_n)$ and using the $1$-Lipschitzness of
$c\mapsto T_c$ with respect to the sup norm, the plug-in error decomposes as
\begin{equation}\label{eq:decomp}
\begin{aligned}
&\E\bigl|D_2^2(\widehat\mu_n,\widehat\nu_n)
  -D_2^2(\mu,\nu)\bigr| \\
&\quad\leq
\underbrace{\E\bigl|E_{\widehat\mu_n,\widehat\nu_n}
  -E_{\mu,\nu}\bigr|}_{\text{marginal term}}
+\underbrace{\E\sup_{a\in\mathcal A}
  \bigl\|c_{a,\widehat\mu_n,\widehat\nu_n}
  -c_{a,\mu,\nu}\bigr\|_\infty}_{\text{cost stability}}
+\underbrace{\E\sup_{a\in\mathcal A}
  \bigl|T_{c_{a,\mu,\nu}}(\widehat\mu_n,\widehat\nu_n)
  -T_{c_{a,\mu,\nu}}(\mu,\nu)\bigr|}_{\text{OT estimation error}}.
\end{aligned}
\end{equation}

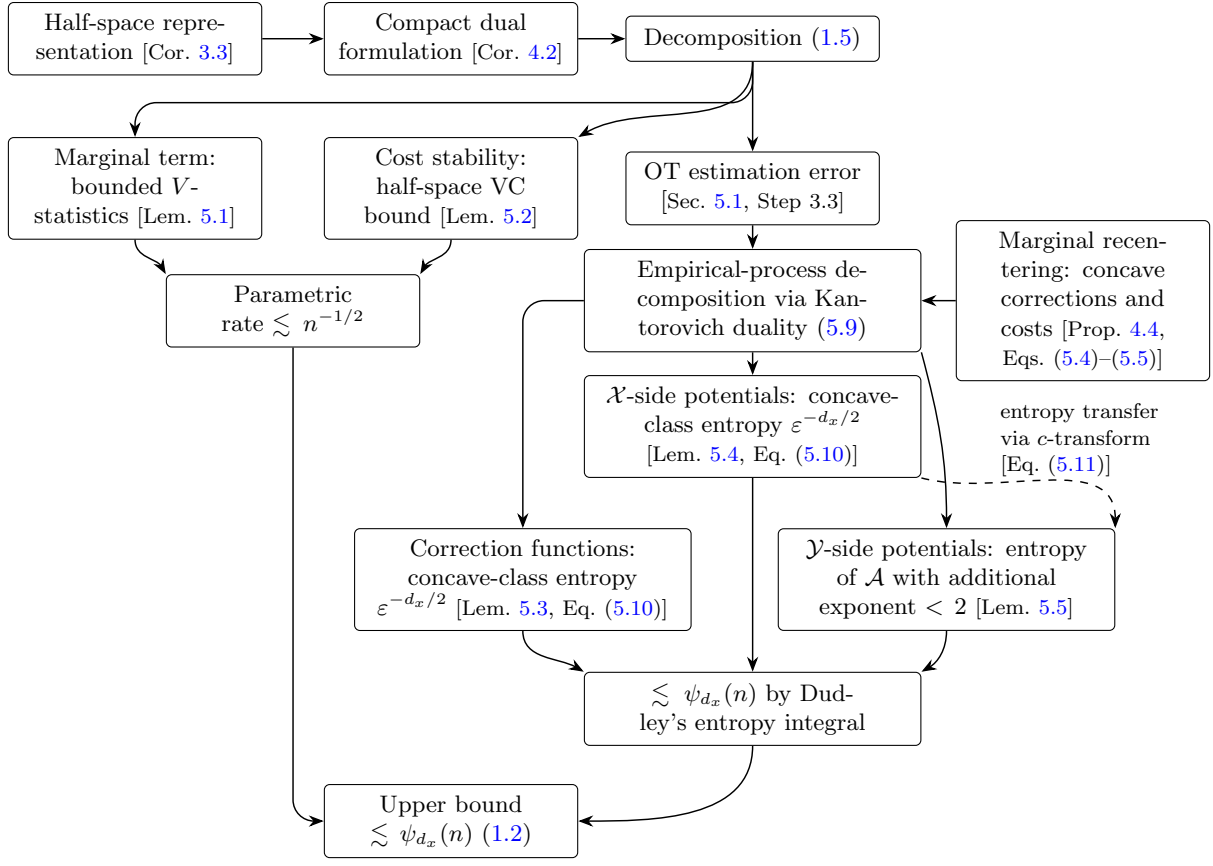
\begin{figure}[t]
\centering
\begin{tikzpicture}[
    x=0.95cm,
    y=0.90cm,
    box/.style={
        draw,
        rounded corners=2pt,
        align=center,
        text width=3.3cm,
        inner xsep=5pt,
        inner ysep=4pt,
        font=\footnotesize
    },
    wide/.style={
        box,
        text width=4.1cm
    },
    small/.style={
        box,
        text width=3.0cm
    },
    flow/.style={
        -{Stealth[length=2mm,width=1.5mm]},
        line width=0.55pt
    },
    transfer/.style={
        flow,
        dashed
    }
]

\node[small] (embedding) at (-4.4,0)
    {Half-space representation
     {\scriptsize [Cor.~\ref{cor:crofton-centered}]}};

\node[small] (duality) at (0,0)
    {Compact dual formulation
     {\scriptsize [Cor.~\ref{cor:compact-duality}]}};

\node[small] (decomposition) at (4.2,0)
    {Decomposition~\eqref{eq:decomp}};

\node[small] (marginal) at (-4.4,-2.2)
    {Marginal term: bounded $V$-statistics
     {\scriptsize [Lem.~\ref{lem:v-statistics}]}};

\node[small] (stability) at (0,-2.2)
    {Cost stability: half-space VC bound
     {\scriptsize [Lem.~\ref{lem:cost-stability}]}};

\node[small] (oterror) at (4.2,-2.2)
    {OT estimation error
     [\scriptsize Sec.~\ref{sec:pf-upper}, Step~3.3]};

\node[small] (parametric) at (-2.2,-4)
    {Parametric rate $\lesssim n^{-1/2}$
     };

\node[small] (recentering) at (8.8,-3.85)
    {Marginal recentering: concave corrections and costs
     {\scriptsize [Prop.~\ref{lem:signed-mixture},
     Eqs.~\eqref{eq:phi}--\eqref{eq:positive-mixture}]}};

\node[wide] (empirical) at (4.2,-3.85)
    {Empirical-process decomposition via Kantorovich duality
     ~\eqref{eq:G22-three-terms}};

\node[wide] (xside) at (4.2,-5.7)
    {$\mathcal X$-side potentials: concave-class entropy
     $\varepsilon^{-d_x/2}$
     {\scriptsize [Lem.~\ref{lem:reg of F},
     Eq.~\eqref{eq:metric entropy conc}]}};

\node[wide] (correction) at (1.0,-7.95)
    {Correction functions: concave-class entropy
     $\varepsilon^{-d_x/2}$
     {\scriptsize [Lem.~\ref{prop:ambient-regularity},
     Eq.~\eqref{eq:metric entropy conc}]}};

\node[wide] (yside) at (6.9,-7.95)
    {$\mathcal Y$-side potentials: entropy of $\mathcal A$ with
     additional exponent $<2$
     {\scriptsize [Lem.~\ref{lem:convex-hull}]}};

\node[wide] (dudley) at (4.2,-9.85)
    {
     $\lesssim\psi_{d_x}(n)$ by Dudley's entropy integral};

\node[small] (conclusion) at (0,-11.5)
    {Upper bound $\lesssim\psi_{d_x}(n)$ \eqref{eq:upper-bound}};

\draw[flow]
    (embedding) -- (duality);

\draw[flow]
    (duality) -- (decomposition);

\draw[flow]
    (decomposition) -- (oterror);

\draw[flow]
    (decomposition.south)
    to[out=-90,in=30]
    (stability.north east);

\draw[flow,rounded corners=8pt]
    (decomposition.south)
    -- ++(0,-0.6)
    -| (marginal.north);

\draw[flow]
    (marginal.south)
    to[out=-90,in=120]
    (parametric.north west);

\draw[flow]
    (stability.south)
    to[out=-90,in=60]
    (parametric.north east);

\draw[flow]
    (oterror) -- (empirical);

\draw[flow]
    (recentering.west) -- (empirical.east);

\draw[flow]
    (empirical) -- (xside);

\draw[flow,rounded corners=8pt]
    (empirical.west)
    -| (correction.north);

\draw[transfer]
    (xside.south east)
    to[out=-15,in=90]
    node[
        pos=0.8,
        above=3pt,
        text width=2.7cm,
        font=\scriptsize,
        align=left
    ]
    {entropy transfer via $c$-transform
     [Eq.~\eqref{eq:metric-entropy-tildeG}]}
    (yside.north east);

\draw[flow]
    (empirical.south east)
    to[out=-75,in=90]
    (yside.north);

\draw[flow]
    (correction.south)
    to[out=-90,in=150]
    (dudley.north west);

\draw[flow]
    (xside.south) -- (dudley.north);

\draw[flow]
    (yside.south)
    to[out=-90,in=30]
    (dudley.north east);

\draw[flow,rounded corners=8pt]
    (parametric.south)
    -- (-2.2,-11.5)
    -- (conclusion.west);

\draw[flow]
    (dudley.south)
    to[out=-90,in=0]
    (conclusion.east);

\end{tikzpicture}
\caption{Proof flow chart for the upper bound  \eqref{eq:upper-bound} (under the assumption $d_{x} = d_x \wedge d_y$). $\psi_{d_x}(n)$ denotes the target rate $n^{-\frac{2}{(d_x\wedge d_y)\vee4}}
(\log n)^{\mathbf 1_{\{d_x\wedge d_y=4\}}}$.}
\label{fig:proof-flow}
\end{figure}

The analysis of the marginal term involves $V$-statistics with bounded kernels (Lemma \ref{lem:v-statistics}), while the
cost stability term can be controlled via a standard VC bound argument (Lemma \ref{lem:cost-stability}).
Both terms have the parametric rate $n^{-1/2}$.

The dominant term is the OT estimation error. By Kantorovich duality,
it can be bounded by suprema of empirical processes over the associated
classes of dual potentials, whose metric entropy determines the rate via Dudley's entropy integral. However, this does not directly yield the $W_2^2$-type target rate due to the absence of some concavity structure, attributed to the signedness of the measure $\theta_{a,y}$ in the representation
\[
c_{a,\mu,\nu}(x,y)
=
M_a(y)+\int_{\mathcal X}\|z-x\|\,d\theta_{a,y}(z)
\]
for each $a\in\mathcal A$ and $y\in\mathcal Y$ (Proposition \ref{lem:signed-mixture}). Nevertheless, we can construct a finite positive measure $\Lambda_{\mu,a}$ for which $|\theta_{a,y}| \leq \Lambda_{\mu,a}$, and adding the correction $\phi_{\mu,a}(x) = -\int_{\mathcal X}\|z-x\|\,d\Lambda_{\mu,a}(z)$ offsets the signedness of $\theta_{a, y}$, after which both $\phi_{\mu,a}$ and the recentered cost $c_{a,\mu,\nu}(\cdot,y) + \phi_{\mu,a}$ extend to concave functions on a closed ball containing $\mathcal X$. Since the correction is a function of $x\in \mathcal{X}$, the correction
changes $T_c(\alpha,\beta)$ only by the marginal integral
$\int\phi_{\mu,a}\,d\alpha$. Thus, the OT estimation error can in turn be bounded by
suprema of empirical processes indexed by the dual potential classes of the
recentered costs and an additional term indexed by the correction
functions (cf. \eqref{eq:G22-three-terms} in Step 3.2 of Section \ref{sec:pf-upper}).

The correction functions and the recentered potentials on the
$\mathcal X$-side are concave with sufficient regularity (Lemmas \ref{prop:ambient-regularity} and \ref{lem:reg of F}), and their metric
entropies scale as $\varepsilon^{-d_x/2}$, which yields the target rate through Dudley's entropy integral. The metric entropy bound for the potential class on the $\mathcal{X}$-side is transferred to the potential class on the higher-dimensional
$\mathcal Y$-side through the $c$-transform \eqref{eq:metric-entropy-tildeG}, which
parallels the lower-complexity adaptation principle for empirical
OT~\cite{groppe2024lower,hundrieser2024empirical}. As a byproduct of this transfer, the potential class carries an additional contribution from the metric entropy of
the parameter class $\mathcal A$. While this contribution may exceed the exponent $d_x/2$
observed for the other classes, the corresponding exponent is always
strictly smaller than $2$, and Dudley's entropy integral still gives the
parametric rate $n^{-1/2}$, which equals the target rate since necessarily
$d_x < 4$ in this case. All the empirical-process terms are therefore of
order at most the target rate in~\eqref{eq:upper-bound}, which absorbs the parametric terms. Figure~\ref{fig:proof-flow} describes a flow chart of where the different tools are introduced and how they are combined to prove Theorem~\ref{thm:main-informal}.

The lower bound~\eqref{eq:lower-bound} follows readily from the corresponding
lower bound argument for the $(2,2)$ case in
\cite[Theorem~4.2]{zhang2024gromov}, together with a comparison inequality
between $D_{2}^2$ and $D_{2,2}^2$.

\subsection{Organization}
The rest of the paper is organized as follows. Section~\ref{sec:preliminaries} reviews the necessary background on OT and Bochner integration. Section~\ref{sec:crofton-halfspace} introduces the half-space representation of the Euclidean distance. Section~\ref{sec:duality} presents the dual formulation for $2$-GW, and the main body concludes with the proof of the main theorem in Section~\ref{sec:proofs-main-results}. Appendices~\ref{sec:crofton-identities}--\ref{sec:sample-proofs} contain proofs of technical results omitted in the main body, and Appendix~\ref{ex:failure-semiconcavity} presents a one-dimensional example showing that the induced costs need not be semiconcave.

\subsection{Notation} Throughout the paper, we fix dimensions $d_x \geq 1$ and $d_y \geq 1$, and for
probability measures $\mu$ on $\R^{d_x}$ and $\nu$ on $\R^{d_y}$, we write
$\widehat\mu_n = \frac 1n \sum_{i=1}^n \delta_{X_i}$ and
$\widehat\nu_n = \frac 1n \sum_{i=1}^n \delta_{Y_i}$ for the empirical
measures constructed from i.i.d. samples $X_1,\ldots,X_n \sim \mu$ and
$Y_1,\ldots,Y_n \sim \nu$, with the two samples independent of each other. We use $D_2=D_{2,1}$ throughout. For nonnegative quantities $A$ and $B$, we write $A \lesssim B$ if
$A \le CB$ for some constant $C > 0$, and
$A \lesssim_{p_1,\ldots,p_k} B$ when $C$ is allowed to depend only on the
parameters $p_1,\ldots,p_k$. For a real Hilbert space $H$, we denote by $\|\cdot\|_H$ and
$\langle \cdot, \cdot \rangle_H$ the corresponding norm and inner product.
For a positive measure $\rho$ on $S$, we write $L^1(S,\rho)$ for the
corresponding scalar $L^1$ space.
When $H = \R^d$, we drop the subscripts and simply write $\|\cdot\|$ and
$\langle \cdot, \cdot \rangle$. We write $B^d_r(x)$ for the closed Euclidean
ball in $\R^d$ of radius $r$ centered at $x$, and $\Sd^{d-1}$ for the unit
sphere in $\R^d$, with $\sigma_{d-1}$ the uniform probability measure on
$\Sd^{d-1}$. Given a topological space $\mathcal{X}$, let
$\mathcal P(\mathcal X)$ be the collection of Borel probability measures on
$\mathcal{X}$. For $p\geq1$, let $\mathcal P_p(\R^d)$ denote the measures
$\mu\in\mathcal P(\R^d)$ satisfying $\int\|x\|^p\,d\mu(x)<\infty$. Given a Borel measure $\mu$ on $\mathcal{X}$ and a Borel
measurable map $T: \mathcal{X} \to \mathcal{Y}$ to another topological space
$\mathcal{Y}$, let $T_\sharp \mu$ denote the pushforward of $\mu$ under $T$.
Furthermore, let $L^2(\mathcal X,\mu)$ be the space of $\mu$-a.e.-equivalence
classes of Borel measurable functions $f:\mathcal X\to\R$ satisfying
$\int_{\mathcal X}|f|^2\,d\mu<\infty$, equipped with the norm
$\|f\|_{L^2(\mathcal X,\mu)}
=(\int_{\mathcal X}|f|^2\,d\mu)^{1/2}$ and inner product
$\langle f,g\rangle_{L^2(\mathcal X,\mu)}
=\int_{\mathcal X}fg\,d\mu$. For a real-valued function $f$ on a set $S$, we write
$\|f\|_{\infty,S} = \sup_{x \in S} |f(x)|$, abbreviated $\|f\|_\infty$ when
$S$ is clear from context. Given a
finite signed Borel measure $\eta$ on $\mathcal X$, we write $|\eta|$ for
its total variation measure and $\|\eta\|_{\TV} = |\eta|(\mathcal X)$ for its
total variation norm. We denote by
$\mathcal{N}(\varepsilon, \mathcal{F}, \|\cdot\|)$ the
$\varepsilon$-covering number of a function class $\mathcal{F}$ with respect
to a norm $\|\cdot\|$. Given $Q \subset \R^d$ and a Lipschitz continuous
function $f:Q\to\R$, let $\|f\|_{\Lip(Q)}$ denote its Lipschitz
constant, abbreviated $\|f\|_{\Lip}$ when $Q$ is clear. For a function on a product space, $\|f\|_{\Lip,x}$ and $\|f\|_{\Lip,y}$ denote the uniform Lipschitz constants in the first and second variables, respectively. For a subset $A$ of a normed vector space, $\operatorname{aco}(A)$ denotes its absolutely convex hull and $\overline{\operatorname{aco}}(A)$ its norm closure. We write $\E_{X\sim \mu} f(X) = \int_\mathcal{X} f\,d\mu$ for $\mu \in \mathcal{P}(\mathcal{X})$. All expectations involving suprema over possibly nonmeasurable function classes are understood as outer expectations.

\subsection{Acknowledgments}
The authors thank Promit Ghosal for suggesting the problem and for valuable feedback on the exposition.

%% file: sections/preliminaries.tex
\section{Preliminaries} \label{sec:preliminaries}
\subsection{Background on optimal transport} 
Let $\mathcal X$ and $\mathcal Y$ be Polish spaces equipped with their Borel
$\sigma$-algebras, and let $\mu \in \mathcal P(\mathcal X)$ and
$\nu \in \mathcal P(\mathcal Y)$. Given a measurable cost function
$c:\mathcal X \times \mathcal Y \to (-\infty,+\infty]$ satisfying
$c(x,y)\geq \ell_{\mathcal X}(x)+\ell_{\mathcal Y}(y)$ for some
integrable functions $\ell_{\mathcal X}\in L^1(\mathcal X,\mu)$ and
$\ell_{\mathcal Y}\in L^1(\mathcal Y,\nu)$, the classical OT problem is
\[
T_c(\mu,\nu)
=
\inf_{\pi\in\Pi(\mu,\nu)}
\int_{\mathcal X \times \mathcal Y} c(x,y)\,d\pi(x,y),
\]
where $\Pi(\mu,\nu)$ is the set of couplings of $\mu$ and $\nu$. In this
formulation, a coupling $\pi$ specifies how mass from $\mu$ is matched with
mass from $\nu$, and the objective minimizes the average cost of doing so.
The objective is linear in $\pi$ and the marginal constraints are linear, so
the problem can be viewed as a linear program over couplings. Under mild assumptions on the cost\footnote{Namely, $c$ is lower
semicontinuous, $T_c(\mu,\nu)<\infty$, and the functions
$\ell_{\mathcal X}$ and $\ell_{\mathcal Y}$
can be chosen upper semicontinuous.}, this linear program admits strong duality, and its
dual, known as the Kantorovich dual problem, is given by
\[
T_c(\mu,\nu)
=
\sup_{\substack{
(\varphi,\psi)\in C_b(\mathcal X)\times C_b(\mathcal Y)\\
\varphi(x)+\psi(y)\leq c(x,y)
}}
\left\{
\int_{\mathcal X}\varphi\,d\mu
+
\int_{\mathcal Y}\psi\,d\nu
\right\}.
\]
Define the $c$-transform of a function
$f:\mathcal X\to\mathbb R\cup\{-\infty\}$ that is not identically
$-\infty$ by
$f^c(y):=\inf_{x\in\mathcal X}\{c(x,y)-f(x)\}$, and analogously for
functions on $\mathcal Y$. The supremum above may then be taken over $c$-conjugate pairs
$(\varphi,\psi)\in L^1(\mathcal X,\mu)\times L^1(\mathcal Y,\nu)$ satisfying
$\varphi^c=\psi$ and $\psi^c=\varphi$
\cite[Thm.~5.10(i)]{villani2009optimal}. If the cost is continuous and bounded,
then the supremum is attained by some pair $(\varphi,\varphi^c)$
\cite[Thm.~5.10(iii)]{villani2009optimal}. Under the boundedness assumption on the cost, we
will often normalize the potentials by requiring
$\sup_{\mathcal X}\varphi=0$. For comprehensive treatments of OT, see
\cite{peyre2019computational,santambrogio2015optimal,villani2009optimal}.

\subsection{Bochner integration}
We recall standard facts concerning Bochner integration of Hilbert-space-valued functions needed to establish a duality theory for $2$-GW. The following discussion is drawn from \cite[Section~1.2]{hytonen2016analysis}.

Let $H$ be a separable Hilbert space, let $S$ be a topological space, and let $(S,\Sigma,\rho)$ be a finite positive Borel measure space. An $H$-valued function $F:S\to H$ is called Bochner integrable with respect to $\rho$ if it is Borel measurable and $\int_S\|F(s)\|_H\,d\rho(s)<\infty$. Its Bochner integral is denoted by $\int_S F(s)\,d\rho(s)\in H$, and we write $L^1(S,\rho;H)$ for the corresponding Bochner space equipped with $\|F\|_{L^1(S,\rho;H)}:=\int_S\|F(s)\|_H\,d\rho(s)$. The integral is constructed in a similar fashion to the Lebesgue integral, relying on the fact that $F$ is a $\rho$-a.e.\ limit of simple functions $F_n=\sum_{k=1}^{m_n}\mathbf{1}_{A_{n,k}}h_{n,k}$, where $A_{n,k}\in\Sigma$ and $h_{n,k}\in H$, whose integrals are defined by $\int_S F_n\,d\rho:=\sum_{k=1}^{m_n}\rho(A_{n,k})h_{n,k}$. The sequence may be chosen so that $F_n\to F$ in $L^1(S,\rho;H)$, and $\int_S F\,d\rho$ is defined as the limit of $\int_S F_n\,d\rho$ in $H$. A useful sufficient condition for Bochner integrability is that $F$ is continuous and bounded, provided that $\rho$ is finite. Moreover, Bochner integrals satisfy the triangle inequality $\left\|\int_S F(s)\,d\rho(s)\right\|_H \leq \int_S\|F(s)\|_H\,d\rho(s).$

For our analysis, we will use the extension of Bochner integration to finite signed measures. If $\rho$ is signed, we consider its Jordan decomposition
$\rho = \rho^{+} - \rho^{-}$, where $\rho^{+}$ and $\rho^{-}$ are positive measures on $S$ with $\rho^{+}\perp \rho^{-}$. We say that $F$ is Bochner integrable with respect to $\rho$ if it is Bochner integrable with respect to both $\rho^{+}$ and $\rho^{-}$, and define its Bochner integral with respect to $\rho$ as
\[
\int_S F(s)\,d\rho(s)
    :=
    \int_S F(s)\,d\rho^+(s)
    -
    \int_S F(s)\,d\rho^-(s).
\]
In this sense, an equivalent condition for Bochner integrability is that  $F$ is Borel measurable and
$\int_S \|F(s)\|_H\,d|\rho|(s)<\infty$. Furthermore, the following fundamental property of the Bochner integral remains valid under this extension.
\begin{lemma}[Bochner integration commutes with bounded linear operators]
\label{lem:Bochner}
Let $F:S\to H$ be Bochner integrable with respect to $\rho$, potentially signed. If $W$ is another separable Hilbert space and $T:H\to W$ is a bounded linear operator, then $TF:S\to W$ is Bochner integrable and $T\int_S Fd\rho=\int_S TFd\rho.$
In particular, for every $h\in H$, $\left\langle \int_S F(s)d\rho(s),h\right\rangle_H=\int_S \langle F(s),h\rangle_H d\rho(s).$
\end{lemma}

In this paper, we repeatedly use Bochner integrals of $L^2$-valued
functions. Although each such integral may equivalently be viewed as the
$L^2$-equivalence class represented by the corresponding pointwise
integral (by Lemma~\ref{lem:Bochner} and Fubini, see for instance the
proof of Lemma~\ref{lem:Mgamma representation}), the Bochner formulation
is more convenient for the metric-entropy analysis, as explained in
Section~\ref{sec:duality}. We therefore use the Bochner formulation for all
$L^2$-valued integrals throughout.

%% file: sections/crofton-2.tex
\section{Half-space representation of the Euclidean distance}
\label{sec:crofton-halfspace}
In this section, we develop a half-space representation of the Euclidean distance. Our representation is based on the Crofton formula. Classically, it refers to a result of Cauchy and Crofton which expresses the arc length of a planar curve in terms of the number of intersection points between the curve and a random line. This can be generalized to higher dimensions, where the volume of a $k$-dimensional submanifold in $\R^n$ is recovered by integrating over the number of intersection points it makes with a family of $(n-k)$-dimensional submanifolds ~\cite{Brothers1966,Chern1952,Federer1954,Howard1993}. For our purposes, we use the Crofton formula in the special case of measuring the length of the line segment between two points. We first state this result in an abstract metric space setting. By \cite[Remark 2.9, Corollary 5.4]{ChatterjiDrutuHaglund2010}, any metric space $(\mathcal{X},d)$ that embeds isometrically into an $L^1$ space admits a measured-wall representation of its metric. Informally, there exists a family of walls $\mathcal W$ equipped with a measure $\Lambda$ such that:
\[
    d(x,x')=\int_{\mathcal W}\mathbf{1}_{\{W\text{ separates }x,x'\}}d\Lambda(W).
\]
Here a wall $W = \{H,H^c\}$ is a partition $\mathcal{X}=H\sqcup H^c$. In this sense, the walls play the role of the random lines in the classical Crofton formula. If we denote by $H_W$ any one side of a wall $W$, we can equivalently write
\begin{equation}\label{eq:general-crofton}
d(x,x')=\int_{\mathcal W}(\mathbf{1}_{H_W}(x) - \mathbf{1}_{H_W}(x'))^2 d\Lambda(W).
\end{equation}
For Euclidean spaces, i.e., $d(x,x') = \|x - x'\|$, the walls are partitions of the space by affine half-spaces \cite{Santalo2004}. Each half-space in $\R^d$ can be parameterized by $\xi = (\theta, t) \in \Sd^{d-1}\times \R$ as $H_\xi = \{x\in \R^d: \theta\cdot x > t\}$, where $\theta$ is the normal direction and $t$ is the offset from the origin. The following theorem formalizes the half-space representation of Euclidean distance by choosing a suitable measure on the parameter space.

\begin{theorem}[Global half-space representation]
\label{theorem:global-crofton}
Let $\mathcal W_d:=\Sd^{d-1}\times\R$ and
$a_d:=\E_{\theta\sim\sigma_{d-1}}|\theta_1|>0$, where
$\sigma_{d-1}$ is the uniform probability measure on $\Sd^{d-1}$.
Let $dt$ be Lebesgue measure on $\R$, and define the
measure $\Lambda_d$ on $\mathcal W_d$ by
$d\Lambda_d(\theta,t):=a_d^{-1}d\sigma_{d-1}(\theta)\,dt$. Then, for
every $x,x'\in\R^d$,
\begin{equation}\label{eq:global-crofton}
\|x-x'\|
=
\int
\bigl(\one_{H_\xi}(x)-\one_{H_\xi}(x')\bigr)^2
\,d\Lambda_d(\xi).
\end{equation}
\end{theorem}

The proof of Theorem~\ref{theorem:global-crofton} is given in Appendix~\ref{pf:global-crofton}, and
is a direct consequence of the rotational invariance of the measure
$\sigma_{d-1}$. Since the measure $\Lambda_d$ is not finite, the mappings $\xi \mapsto \one_{H_\xi}(x)$ do not belong to $L^2(\mathcal{W}_d, \Lambda_d)$. Nevertheless, centering these functions with
respect to a probability measure $\mu \in \mathcal{P}_1(\R^d)$ with finite first moment, we obtain the following squared $L^2$-representation of the Euclidean distance.

\begin{corollary}[Global centered $L^2$ embedding]
\label{cor:global-crofton-centered}
Let $\mu\in\mathcal P_1(\R^d)$. For each $x\in \R^{d}$ and $\xi \in \mathcal{W}_d$, define
\begin{equation}\label{eq:global-centered-crofton}
\widetilde U_x^\mu(\xi)
=
\one_{H_\xi}(x)-\mu(H_\xi).
\end{equation}
Then $\widetilde U_x^\mu\in L^2(\mathcal W_d,\Lambda_d)$ for every
$x\in\R^d$, and 
\begin{equation}\label{eq:global-centeredl2embedding}
\|x-x'\|
=
\|\widetilde U_x^\mu-\widetilde U_{x'}^\mu\|^2_{L^2(\mathcal{W}_d, \Lambda_d)}.
\end{equation}
Furthermore, the map $x\mapsto \widetilde U_x^\mu$ is continuous and Bochner
integrable with respect to $\mu$, and
$\int \widetilde U_x^\mu\,d\mu(x)=0$ in $L^2(\mathcal{W}_d, \Lambda_d)$.
\end{corollary}
The proof of Corollary~\ref{cor:global-crofton-centered} is given in Appendix~\ref{pf:global-crofton-centered}. The centering identity $\int \widetilde U_x^\mu\,d\mu(x)=0$ plays a key role in developing duality theory for the $2$-GW distance.
\subsection{Specialization to compact subsets} When the marginals are compactly supported, which is of interest in our sample complexity analysis, it suffices to
consider the Euclidean distance between points of a compact subset
$\mathcal{X} \subset B_R^d(0)$ for some $R > 0$ large enough. Since an affine hyperplane $\partial H_\xi = \{x\in \R^d: \theta\cdot x = t\}$ intersects $B^d_R(0)$ if and only if $t \in [-R, R]$, we may restrict our parameter space for the collection of half-spaces to $\mathcal{W}_\mathcal{X} = \Sd^{d-1} \times [-R,R]$. In this case, the underlying measure $\Lambda_{d}$ can be replaced by a finite measure upon restricting the Lebesgue measure $dt$ to $[-R,R]$. Moreover, we may replace each $H_\xi$ by its
restriction $H_\xi \cap \mathcal{X}$, and in what follows $H_\xi$ denotes this restriction. Figure~\ref{fig:euclidean-crofton-halfspace} provides an illustration of this parameterization within the ambient ball \(B_R^d(0)\).

\begin{figure}[h]
\centering
\begin{tikzpicture}[scale=1, font=\small, >=Latex]
    \coordinate (O) at (0,0);
    \coordinate (A) at (-1.42,2.81);
    \coordinate (B) at (2.30,-2.15);
    \coordinate (P) at (0.44,0.33);
    \fill[gray!4] (O) circle (2.25);
    \begin{scope}
        \clip (O) circle (2.25);
        \fill[gray!10]
            (A) -- (B) -- (4.2,-0.55) -- (0.6,4.25) -- cycle;
    \end{scope}
    \fill[gray!22]
        plot[smooth cycle, tension=.85] coordinates {
            (-1.05,-0.55) (-0.35,-1.05) (0.75,-0.75) (1.35,0.05)
            (0.78,0.93) (-0.62,0.83) (-1.28,0.12)
        };
    \begin{scope}
        \clip (O) circle (2.25);
        \fill[
            pattern={Lines[angle=45,distance=1.5pt,line width=0.4pt]},
            pattern color=gray!60
        ]
            (A) -- (B) -- (4.2,-0.55) -- (0.6,4.25) -- cycle;
    \end{scope}
    \draw[gray!60, thick] (O) circle (2.25);

    \draw[gray!45, thin]
        plot[smooth cycle, tension=.85] coordinates {
            (-1.05,-0.55) (-0.35,-1.05) (0.75,-0.75) (1.35,0.05)
            (0.78,0.93) (-0.62,0.83) (-1.28,0.12)
        };
    \node[gray] at (-1.8,1.9) {$B_R^d(0)$};
    \node at (-0.7,0.5) {$\mathcal X$};
    \draw[thick] (A) -- (B);
    \node[below right] at (B) {$\partial H_\xi$};
    \draw[gray!75, line width=0.5pt, dash pattern=on 1pt off 1pt] (O) -- (P)
    	node[pos=.55, above left=1pt, yshift=-3pt, xshift=3pt] {$t$};

    \fill (O) circle (1.2pt);
    \node[below left] at (O) {$0$};
    \draw[-{Latex[length=2.4mm,width=1.2mm]}, thick] (P) -- ++(0.65,0.49);
    \node at (1.18,0.98) {$\theta$};
    \fill (0.92,0.20) circle (2pt);
    \node[above right] at (0.92,0.20) {$x$};

    \fill (-0.72,-0.25) circle (2pt);
    \node[below left] at (-0.6,-0.25) {$x'$};
    \node at (1.45,1.35) {$H_\xi$};
\end{tikzpicture}
\caption{
A schematic of the affine half-space
$H_\xi=\{z:\theta\cdot z>t\}$.
The boundary $\partial H_\xi=\{z:\theta\cdot z=t\}$ is perpendicular to the
normal direction $\theta$, and the hatched region represents
$H_\xi\cap B_R^d(0)$.
The dashed segment shows the signed offset $t$ from the origin to the wall
in the schematic case $t>0$.
}
\label{fig:euclidean-crofton-halfspace}
\end{figure}
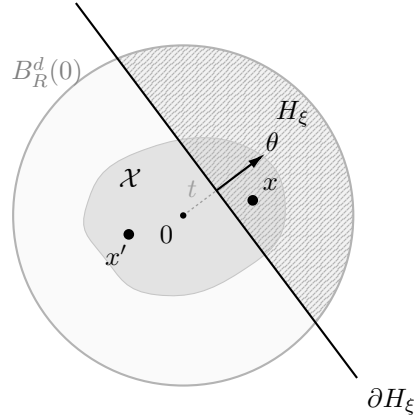

\begin{corollary}[Compact half-space representation]\label{thm:euclidean-crofton}\label{cor:crofton-centered} Let $\mathcal{X} \subset B^d_R(0)$ be a compact subset of $\R^d$, and define the probability measure $\lambda_{\mathcal{X}}$ on $\mathcal{W}_{\mathcal{X}} = \Sd^{d-1} \times [-R,R]$ by $d\lambda_\mathcal{X}(\theta,t) = \frac{1}{2R}d\sigma_{d-1}(\theta) dt$. Then for any $x,x' \in \mathcal{X}$,
    \begin{equation}
    \label{eq:euclidean-crofton}
        \|x-x'\|
        =\kappa_{\mathcal{X}}\int_{\mathcal{W}_{\mathcal{X}}}
        \left(\mathbf 1_{H_\xi}(x)-\mathbf 1_{H_\xi}(x')\right)^2
        d\lambda_{\mathcal{X}}(\xi),
    \end{equation} 
    where $\kappa_{\mathcal{X}} = \frac{2R}{a_d}$. Consequently, if $\mu\in\mathcal P(\mathcal X)$ and
$U^\mu:\mathcal X\to
L^2(\mathcal W_{\mathcal X},\lambda_{\mathcal X})$ is given by
\begin{equation}\label{eq:centered-crofton}
U_x^\mu(\xi)
=
\sqrt{\kappa_{\mathcal X}}
\bigl(\one_{H_\xi}(x)-\mu(H_\xi)\bigr),
\end{equation}
then for every $x,x' \in \mathcal{X}$,
\begin{equation}\label{eq:centeredl2embedding}
\|x-x'\|
=
\|U_x^\mu-U_{x'}^\mu\|^2_{L^2(\mathcal{W}_\mathcal{X}, \lambda_\mathcal{X})}.
\end{equation}
Moreover, $\sup_{x\in\mathcal X}\|U_x^\mu\|^2_{L^2(\mathcal{W}_\mathcal{X}, \lambda_\mathcal{X})}\leq
\kappa_{\mathcal X}$ and 
$\int U_x^\mu\,d\mu(x)=0$.
\end{corollary}

The proof is omitted as the claims readily follow from the proofs of Theorem~\ref{theorem:global-crofton} and Corollary~\ref{cor:global-crofton-centered}. The underlying measure $\lambda_\mathcal{X}$ is normalized, as the constants can be handled nicely in the subsequent analysis. Exploiting the finiteness of the measure $\lambda_\mathcal{X}$, we obtain the following identity that is useful when expanding out \eqref{eq:centeredl2embedding}.
\begin{lemma} \label{lem:crofton-overlap}
    For any $x,x'\in \mathcal{X}$,
    \begin{equation} \label{eq:crofton-overlap}
        \int_{\mathcal{W}_{\mathcal{X}}}
        \mathbf 1_{H_\xi}(x)\mathbf 1_{H_\xi}(x')
        d\lambda_{\mathcal{X}}(\xi)
        =\frac12-\frac{\|x-x'\|}{2\kappa_{\mathcal X}}.
    \end{equation}
\end{lemma}
See Appendix~\ref{pf:crofton-overlap} for the proof. Heuristically, \eqref{eq:crofton-overlap} relates the probability that a random half-space contains both points $x,x' \in \mathcal{X}$ to their Euclidean distance. Finally, we observe several consequences of \eqref{eq:crofton-overlap} that will be useful for our analysis in the compact setting. The proofs of Lemmas~\ref{lem:linear-distance} and~\ref{lem:linear-distance-y} are given in Appendix~\ref{pf:linear-distance} and Appendix~\ref{pf:linear-distance-y}, respectively.

\begin{lemma}\label{lem:linear-distance}
Let $\mu\in\cP(\mathcal X)$. Then for every $x\in \mathcal{X}$,
\begin{equation}\label{eq:linear U}
\|U_x^\mu\|_{L^2(\mathcal{W}_\mathcal{X},\lambda_\mathcal{X})}^2=b_\mu+\int \|z-x\|d\mu(z),
\end{equation}
where $b_\mu =
\kappa_{\mathcal X}
\left(
\int\mu(H_\xi)^2\,d\lambda_{\mathcal X}(\xi)-\frac12
\right)$ is a constant independent of $x$.
Consequently, $x\mapsto \|U_x^\mu\|^{2}_{L^2(\mathcal{W}_\mathcal{X},\lambda_{\mathcal{X}})}$ is $1$-Lipschitz.

\end{lemma}

\begin{lemma}\label{lem:linear-distance-y}
Let $\mu\in\cP(\mathcal X)$. For $u,x\in\mathcal X$, define
\[
\alpha_u^\mu(x)
:=
\frac{1}{\sqrt{\kappa_{\mathcal X}}}
\int\1_{H_\xi}(u)U_x^\mu(\xi)\,d\lambda_{\mathcal X}(\xi).
\]
Then $|\alpha_u^\mu(x)|\leq\frac12$ for all $u,x\in\mathcal X$, and
\begin{equation}\label{eq:linear-alpha}
\alpha_u^\mu(x)
=
A_\mu(u)-\frac{\|u-x\|}{2\kappa_{\mathcal X}},
\end{equation}
where
$A_\mu(u)
=
\frac12-\int\1_{H_\xi}(u)\mu(H_\xi)\,d\lambda_{\mathcal X}(\xi).
$
Consequently, $x\mapsto \alpha_u^\mu(x)$ is
$1/(2\kappa_{\mathcal X})$-Lipschitz.
\end{lemma}

%% file: sections/dual.tex
\section{Duality} \label{sec:duality}

In this section, we leverage the squared $L^2$ representations of the
Euclidean distance obtained in Section~\ref{sec:crofton-halfspace} to
establish duality results for $D_2^2$. Using the global version from Corollary~\ref{cor:global-crofton-centered}, we begin
with a general duality result for marginals that are not necessarily
compactly supported. Note that in the following, we do not assume
that the marginals are centered (see Remark \ref{rem:centering} for details).

\begin{theorem}[Duality, global version]
\label{thm:global-duality}
Let $(\mu, \nu)\in\mathcal P_2(\R^{d_x})\times \mathcal P_2(\R^{d_y})$.\footnote{We impose the finite-second-moment assumption to ensure that $D_2(\mu,\nu)$ is finite without the compactness assumption.} Let
$\widetilde U^\mu:\R^{d_x}\to L^2(\mathcal W_{d_x},\Lambda_{d_x})$ and
$\widetilde V^\nu:\R^{d_y}\to L^2(\mathcal W_{d_y},\Lambda_{d_y})$ be the
global centered embeddings from \eqref{eq:global-centered-crofton} for $\R^{d_x}$ and
$\R^{d_y}$, respectively. Set
$(\widetilde{\mathcal W},\widetilde\Lambda)
:=
(\mathcal W_{d_x}\times\mathcal W_{d_y},
\Lambda_{d_x}\otimes\Lambda_{d_y})$, and define
$m_\mu:=\int\|\widetilde U_x^\mu\|_{L^2(\mathcal W_{d_x},\Lambda_{d_x})}^2
\,d\mu(x)$ and
$m_\nu:=\int\|\widetilde V_y^\nu\|_{L^2(\mathcal W_{d_y},\Lambda_{d_y})}^2
\,d\nu(y)$.
Then
\begin{equation}\label{eq:global-duality}
D_2^2(\mu,\nu)
=
E_{\mu,\nu}
+
\inf_{a\in \mathcal B_{\mu, \nu}}
\left\{
\frac18\|a\|_{L^2(\widetilde{\mathcal W},\widetilde\Lambda)}^2
+
T_{\widetilde c_{a,\mu,\nu}}(\mu,\nu)
\right\},
\end{equation}
where 
\begin{align*}
E_{\mu,\nu}:=
\iint\|x-x'\|^2\,d\mu(x)\,d\mu(x')
+
&\iint\|y-y'\|^2\,d\nu(y)\,d\nu(y')\\
&\quad-
\left(
\iint\|x-x'\|\,d\mu(x)\,d\mu(x')
\right)
\left(
\iint\|y-y'\|\,d\nu(y)\,d\nu(y')
\right),
\end{align*}
\begin{equation*}
    \mathcal B_{\mu,\nu}
:=
\left\{
a\in L^2(\widetilde{\mathcal W},\widetilde\Lambda):
\|a\|_{L^2(\widetilde{\mathcal W},\widetilde\Lambda)}
\leq 8\sqrt{m_\mu m_\nu}
\right\},
\end{equation*}
and for each $a\in \mathcal{B}_{\mu,\nu}$, the associated cost
$\widetilde c_{a,\mu,\nu}:\R^{d_x}\times\R^{d_y}\to\R$ is
\begin{equation}\label{eq:global-dual-cost}
\widetilde c_{a,\mu,\nu}(x,y)
=
-4\|\widetilde U_x^\mu\|^2_{L^2(\mathcal W_{d_x},\Lambda_{d_x})}
\|\widetilde V_y^\nu\|^2_{L^2(\mathcal W_{d_y},\Lambda_{d_y})}
-
2\int
a(\xi,\zeta)\widetilde U_x^\mu(\xi)\widetilde V_y^\nu(\zeta)
\,d\widetilde\Lambda(\xi,\zeta).
\end{equation}
\end{theorem}

In \eqref{eq:global-duality},
the coupling-independent term is isolated in $E_{\mu,\nu}$, and the
coupling-dependent part is expressed as a variational problem whose
objective is the OT problem $T_{\widetilde c_{a,\mu,\nu}}$ parametrized by $a\in \mathcal B_{\mu,\nu}$ with an additional
$L^2$ regularization term. The dependence on the coupling now resides
in $T_{\widetilde c_{a,\mu,\nu}}$ and is no longer quadratic. The price of this
linearization via the squared $L^2$ representation~\eqref{eq:global-centeredl2embedding}
is that the parameter set $\mathcal{B}_{\mu,\nu}$ lies in the
infinite-dimensional space $L^2(\widetilde{\mathcal{W}}, \widetilde\Lambda)$, where
boundedness alone does not ensure finite metric entropy. This global formulation is
therefore not sufficient for our sample complexity analysis,
as the metric entropy of the parameter class contributes to the final
empirical process estimates (cf. Step~3.3.3 in the proof of Theorem~\ref{thm:main-informal} in Section~\ref{sec:pf-upper}). 

In contrast, if the marginals $\mu$ and $\nu$ are compactly supported, this obstruction can be avoided. Since the proof of Theorem~\ref{thm:global-duality}, an overview of which is given below, does not rely on the underlying measure spaces for the $L^2$ spaces in its algebraic steps, the compact version of the $L^2$ representation in Corollary~\ref{cor:crofton-centered} may be used instead. In this case,  the parameter class $\mathcal{B}_{\mu,\nu}$ can be replaced by another class which depends only on the supports rather than on the marginals themselves, as the following shows.

\begin{corollary}[Duality, compact version]\label{cor:compact-duality}
Let $(\mu, \nu) \in \mathcal{P}(\mathcal{X})\times \mathcal{P}(\mathcal{Y})$,
where $\mathcal{X} \subset \mathbb{R}^{d_x}$ and
$\mathcal{Y} \subset \mathbb{R}^{d_y}$ are compact. Let
$U^\mu \colon \mathcal{X} \to L^2(\mathcal W_\mathcal{X}, \lambda_\mathcal{X})$ and
$V^\nu \colon \mathcal{Y} \to L^2(\mathcal W_\mathcal{Y}, \lambda_\mathcal{Y})$ be the
centered embeddings from \eqref{eq:centered-crofton} for $\mathcal{X}$ and
$\mathcal{Y}$, respectively, and set
$(\mathcal W, \lambda) := (\mathcal W_\mathcal{X} \times \mathcal W_\mathcal{Y},
\lambda_\mathcal{X} \otimes \lambda_\mathcal{Y})$. For each
$u \in \mathcal{X}$ and $v \in \mathcal{Y}$, define
$h_{u,v} \in L^2(\mathcal W, \lambda)$ by
$h_{u,v}(\xi, \zeta) = \mathbf{1}_{H_\xi}(u)\mathbf{1}_{G_\zeta}(v)$, where
$H_\xi$ and $G_\zeta$ are the affine half-spaces in $\mathbb{R}^{d_x}$ and
$\mathbb{R}^{d_y}$ parametrized by $\xi$ and $\zeta$, respectively (restricted to $\mathcal X$ and $\mathcal{Y}$). Then
\begin{equation}\label{eq:21dual}
D_2^2(\mu, \nu)
= E_{\mu,\nu}
+ \inf_{a \in \mathcal{A}}
\left\{ \frac{1}{8}\|a\|_{L^2(\mathcal W,\lambda)}^2 + T_{c_{a,\mu,\nu}}(\mu, \nu) \right\},
\end{equation}
where $E_{\mu,\nu}$ is as in Theorem~\ref{thm:global-duality},
\begin{equation*}
\mathcal{A} := \left\{ 8\sqrt{\kappa_\mathcal{X}\kappa_\mathcal{Y}}
\int h_{u,v} \, d\eta(u,v) : \eta \text{ signed and }
\|\eta\|_{\TV} \leq 2 \right\} \subset L^2(\mathcal W, \lambda),
\end{equation*}
and, for each $a \in \mathcal{A}$, the associated cost
$c_{a,\mu,\nu} \colon \mathcal{X} \times \mathcal{Y} \to \mathbb{R}$ is
\begin{equation}\label{eq:cost-original}
c_{a,\mu,\nu}(x, y)
= -4 \|U^\mu_x\|_{L^2(\mathcal W_\mathcal{X},\lambda_\mathcal{X})}^2
\|V^\nu_y\|_{L^2(\mathcal W_\mathcal{Y},\lambda_\mathcal{Y})}^2
- 2 \int a(\xi, \zeta) U^\mu_x(\xi) V^\nu_y(\zeta) \, d\lambda(\xi, \zeta).
\end{equation}
\end{corollary}

In the definition of the new parameter class $\cA$, the measure $\eta$ ranges over all finite
signed Borel measures on $\mathcal X\times\mathcal Y$ with total
variation $|\eta|(\mathcal X\times\mathcal Y)\leq2$, and
$\int h_{u,v}\,d\eta(u,v)$ is the Bochner integral.\footnote{Alternatively, in view of
Lemma~\ref{lem:Bochner} and Fubini, it admits the pointwise
representative $
\left(\int h_{u,v}\,d\eta(u,v)\right)(\xi,\zeta)
=
\eta(H_\xi\times G_\zeta)$ for $\lambda$-almost every $(\xi,\zeta)\in\mathcal W$.} 
By the triangle inequality and the bounds
$\|h_{u,v}\|_\infty\leq1$ and $\|\eta\|_{\TV}\leq2$, $\cA$ is
uniformly bounded in both $\|\cdot\|_\infty$ and
$\|\cdot\|_{L^2(\mathcal W,\lambda)}$ by
$16\sqrt{\kappa_{\mathcal X}\kappa_{\mathcal Y}}$. Although $\mathcal{A}$ is a subset of $L^2(\mathcal{W}, \lambda)$, its metric entropy can be controlled. By the
definition of the Bochner integral, $\mathcal{A}$ is contained (up to scaling) in the closed absolutely convex hull
of the collection $\{h_{u,v}\}_{(u,v)\in\mathcal{X}\times\mathcal{Y}}$. The metric entropy of this hull can be estimated quantitatively
(see Appendix~\ref{pf:convex-hull}), which can then be used to obtain relevant empirical process estimates in our
sample-complexity analysis.

The proofs of Theorem~\ref{thm:global-duality}
and Corollary~\ref{cor:compact-duality} are given in
Appendix~\ref{appendix:proofs-duality}. Both follow the same expansion, and we
present here an overview. Directly expanding $D_2^2$ isolates the
coupling-dependent component
\[
\inf_{\gamma\in\Pi(\mu,\nu)}
-2\iint
\|x-x'\|\|y-y'\|
\,d\gamma(x,y)\,d\gamma(x',y').
\]
The global centered embedding \eqref{eq:global-centeredl2embedding} from
Corollary~\ref{cor:global-crofton-centered} and its compact counterpart
\eqref{eq:centeredl2embedding} allow us to replace the Euclidean
distances by squared $L^2$
distances. Expanding them further using the corresponding centering identities, the quadratic dependence on $\gamma$ can be
identified as the negative squared norm $-8\|\widetilde M_\gamma\|^2_{L^2(\widetilde{\mathcal W},\widetilde\Lambda)}$ in the global version, where
$\widetilde M_\gamma:=\int \widetilde U_x^\mu\otimes \widetilde V_y^\nu\,d\gamma(x,y)$, and as $-8\|M_\gamma\|^2_{L^2(\mathcal W,\lambda)}$ in the compact version, where
$M_\gamma:=\int U_x^\mu\otimes V_y^\nu\,d\gamma(x,y)$. Each term is linearized by introducing an auxiliary
parameter in the corresponding product $L^2$ space and optimizing over
it, whose optimizer is $8\widetilde M_\gamma$ or $8M_\gamma$, respectively. The result follows upon restricting the optimization to the corresponding constraint $\mathcal B_{\mu, \nu}$ or $\mathcal{A}$, which follows from the Cauchy--Schwarz inequality for the global version, and the identification
$M_\gamma
=
\sqrt{\kappa_{\mathcal X}\kappa_{\mathcal Y}}
\int h_{x,y}\,d(\gamma-\mu\otimes\nu)(x,y)$ for the compact version.

\subsection{Structure of the induced costs}
\label{sec:cost-structure}
In the compactly supported setting, each $a\in\mathcal A$ can be represented by a signed measure $\eta$ on
$\mathcal X\times\mathcal Y$ with $\|\eta\|_{\TV}\leq2$ via
$a=a_\eta:=8\sqrt{\kappa_\mathcal X\kappa_\mathcal Y}\int h_{u,v}\,d\eta(u,v)$.
With this representation, we obtain the following alternative form of the
cost $c_{a,\mu,\nu}$ expressed in terms of $\eta$.
\begin{proposition} \label{prop:cost-alternate}
    Fix $\mu\in\cP(\mathcal X)$ and $\nu\in\cP(\mathcal Y)$. Let $\eta$ be a signed measure on $\mathcal X\times \mathcal Y$ with $\|\eta\|_{\TV}\le 2$, and set $a_{\eta} = 8\sqrt{\kappa_\mathcal{X}\kappa_\mathcal{Y}}\int h_{u,v} d\eta(u,v)$. Then
\begin{equation} \label{eq:cost-alternate}
c_{a_\eta,\mu,\nu}(x,y)=
-4\|U_x^\mu\|^2_{L^2(\mathcal{W}_\mathcal{X}, \lambda_\mathcal{X})}\|V_y^\nu\|^2_{L^2(\mathcal{W}_\mathcal{Y}, \lambda_\mathcal{Y})}
-
16\kappa_\mathcal{X}\kappa_\mathcal{Y}
\int
\alpha_u^\mu(x)\beta_v^\nu(y) d\eta(u,v),
\end{equation}
where $\alpha_u^\mu(x)=\frac{1}{\sqrt{\kappa_{\mathcal{X}}}}\int \1_{H_\xi}(u)U_x^\mu(\xi)d\lambda_\mathcal{X}(\xi)$, $\beta_v^\nu(y)=\frac{1}{\sqrt{\kappa_{\mathcal{Y}}}}\int\1_{G_\zeta}(v)V_y^\nu(\zeta)d\lambda_\mathcal{Y}(\zeta)$, and the integral in \eqref{eq:cost-alternate} is a scalar integral defined for each $(x,y) \in \mathcal{X} \times \mathcal{Y}$.
\end{proposition}
The proof is a direct application of Lemma \ref{lem:Bochner} and Fubini and is therefore omitted. This form is used to obtain regularity and boundedness of the costs $\{c_{a,\mu,\nu}:a\in\mathcal A\}$ (see Lemma~\ref{prop:ambient-regularity}).

As explained in the proof overview, the main obstacle in the sample
complexity analysis is that the induced costs are not generally semiconcave in either
argument, a simple one-dimensional example of which is given in
Appendix~\ref{app:semiconcavity}. To recover the target rate, we rely on the
following form, which cleanly isolates the dependence of the cost on the
$x$-variable.

\begin{proposition}\label{lem:signed-mixture}
Fix $\mu\in\cP(\mathcal X)$ and $\nu\in\cP(\mathcal Y)$, and let $\eta$ and $a_\eta$ be defined as above. Then for each $y\in \mathcal Y$, there exists an $x$-independent scalar $M_\eta(y) \in \R$ such that for any $x\in \mathcal{X}$,
\begin{equation}\label{eq:signed-mixture}
c_{a_\eta,\mu,\nu}(x,y)=M_\eta(y)+\int \|z-x\|d\theta_{\eta,y}(z),
\end{equation}
where $\theta_{\eta,y}$ is a signed Borel measure on $\mathcal{X}$ given by
\begin{equation}\label{eq:theta-def}
\theta_{\eta,y}:=-4\|V_y^\nu\|^2_{L^2(\mathcal{W}_\mathcal{Y},\lambda_\mathcal{Y})}\mu
+8\kappa_{\mathcal Y}(\pi_1)_\sharp\bigl(\beta_v^\nu(y)d\eta(u,v)\bigr),
\end{equation}
and $\pi_1:\mathcal X\times \mathcal Y\to \mathcal X$ is the projection onto the first component.  Moreover, for a nonnegative Borel measure $\Lambda_{\mu, \eta}=4\kappa_{\mathcal Y}\mu+4\kappa_{\mathcal Y}(\pi_1)_\sharp|\eta|$ on $\mathcal{X}$, it holds that $\Lambda_{\mu,\eta}(\mathcal X)\le 12\kappa_{\mathcal Y}$ and 
$|\theta_{\eta,y}|
\le \Lambda_{\mu,\eta}.$
\end{proposition}

The proof is given in Appendix~\ref{pf:signed-mixture} and combines \eqref{eq:cost-alternate} with the identities \eqref{eq:linear U} and \eqref{eq:linear-alpha} from Lemmas \ref{lem:linear-distance} and \ref{lem:linear-distance-y}. The $y$-variable analogue of Proposition \ref{lem:signed-mixture} follows by symmetry, and we omit the
formalities for brevity.
This exposes the obstruction to the concavity structure in the cost,
stemming from the possible signedness of the measure $\theta_{\eta,y}$. Nevertheless,
since $\theta_{\eta,y}$ is dominated by the positive measure $\Lambda_{\mu,\eta}$,
this form allows the marginal recentering argument given in Step 3.1 in the proof of
Theorem~\ref{thm:main-informal} in Section \ref{sec:pf-upper}, through which the desired entropy bounds for
the dual potential classes can be attained.
\subsection{Further remarks}
We close this section with several remarks on the role of centering in our dual formulation, the relation between our formulation and existing variational formulations of GW distances, and a potential extension of our linearization approach to general even--odd exponent pairs.

\begin{remark}[Centering and empirical measures]
\label{rem:centering}
Although there is typically no loss of generality in assuming that $\mu$ and $\nu$ are centered, i.e.,
$\int_{\mathcal X}x\,d\mu(x)=0$ and
$\int_{\mathcal Y}y\,d\nu(y)=0$, by the isometric
invariance property of GW (see, e.g.,~\cite[Theorem~5.1(a)]{memoli2011gromov}), our dual formulations do not require this
assumption. Instead, the centering is performed at the level of the
embeddings of points into the relevant $L^2$ spaces as in
\eqref{eq:global-centered-crofton} and
\eqref{eq:centered-crofton}. In particular, \eqref{eq:21dual} applies
directly to the empirical measures $\widehat\mu_n$ and
$\widehat\nu_n$, which need not be centered even when the true
marginals are. Statistically, this yields a difference from the dual
formulation of the quadratic $(2,2)$-GW distance $D^{2}_{2,2}$ in
Zhang et al.~\cite{zhang2024gromov}. Their decomposition of
$D^{2}_{2,2}$ into marginal and coupling terms is derived under the
assumption that the marginals are centered
(cf.~\cite[Section~5.1]{zhang2024gromov}), so an additional centering-bias
estimate is required when applying their formulation to empirical
measures
(cf.~\cite[Section~5.3.1]{zhang2024gromov}).
This estimate can be avoided entirely in our sample-complexity
analysis.
\end{remark}

\begin{remark}[Comparison of the linearization to the $(2,2)$-GW distance]
\label{rem:comparison-22}
Despite the different treatment of centering discussed in
Remark~\ref{rem:centering}, our linearization process and that of the
quadratic $(2,2)$-GW distance $D^{2}_{2,2}$ in Zhang et
al.~\cite{zhang2024gromov} are structurally similar. In their
formulation, the coupling-dependent component of $D^2_{2,2}$ is
linearized by introducing a finite-dimensional auxiliary matrix
$A\in\R^{d_x\times d_y}$, yielding the variational expression
$$
\inf_{A\in\R^{d_x\times d_y}}
\bigl\{32\|A\|_F^2+T_{c_A}(\mu,\nu)\bigr\},
\qquad
c_A(x,y)
=
-4\|x\|^2\|y\|^2-32 x^T Ay,
$$
where $\|\cdot\|_F$ is the Frobenius norm, and the optimal auxiliary
parameter for each fixed coupling $\gamma$ is a multiple of the
cross-correlation $\int xy^T\,d\gamma(x,y)$. Our variational
expression has essentially the same structure, with $x$ and $y$
replaced by the $L^2$ embeddings $U_x^\mu$ and $V_y^\nu$: the matrix
$A$ becomes an element $a$ of the corresponding infinite-dimensional
$L^2$ space, the Euclidean and Frobenius norms are replaced by the
$L^2$ norms, the bilinear term $x^\top Ay$ becomes
$\int aU_x^\mu(\xi)V_y^\nu(\zeta)\,d\lambda$, and the optimal
parameter is a multiple of the half-space correlation element
$M_\gamma$. The reason for this correspondence is that replacing the
Euclidean distance by its squared $L^2$ representation raises the
inner exponent from $1$ to $2$, so that our algebraic steps parallel
those of the $(2,2)$ case. Our linearization process may therefore be
understood as an infinite-dimensional analogue of theirs.
\end{remark}

\begin{remark}[Relation to other Hilbert-space formulations and computation]
For compactly supported marginals, \cite{houry2026gromov} derives a Hilbert-space dual formulation for entropic GW in which the Euclidean distances are replaced by costs $c_{\mathcal X}$ and $c_{\mathcal Y}$ of definite conditionally negative type. Such costs admit a representation $c_{\mathcal X}(x,x')=\|\phi(x)-\phi(x')\|_{\mathcal H_{\mathcal X}}^2$ and $c_{\mathcal Y}(y,y')=\|\psi(y)-\psi(y')\|_{\mathcal H_{\mathcal Y}}^2$ for (centered) Hilbert-space-valued maps $\phi:\mathcal X\to\mathcal H_{\mathcal X}$ and $\psi:\mathcal Y\to\mathcal H_{\mathcal Y}$. Denoting the resulting objective by $\operatorname{GW}_{\varepsilon}(\mu,\nu)$, they obtain the variational formulation \cite[Theorem 3.2, Remark C.3]{houry2026gromov}
\[
\operatorname{GW}_{\varepsilon}(\mu,\nu)
=
C_{\mu,\nu}
+
\inf_{\Gamma\in\operatorname{HS}(\mathcal H_{\mathcal Y},\mathcal H_{\mathcal X})}
\left\{
8\|\Gamma\|_{\operatorname{HS}}^2
+
T_{\varepsilon,c_\Gamma}(\mu,\nu)
\right\},
\]
where $C_{\mu,\nu}$ is a coupling-independent term, $\operatorname{HS}(\mathcal H_{\mathcal Y},\mathcal H_{\mathcal X})$ is the set of Hilbert--Schmidt operators $\Gamma: \mathcal H_{\mathcal Y} \to \mathcal H_{\mathcal X}$, and $\|\cdot\|_{\operatorname{HS}}$ and $\langle \cdot, \cdot \rangle_{\operatorname{HS}}$ are respectively the Hilbert--Schmidt norm and inner product. $T_{\varepsilon,c}$ denotes entropic OT with cost $c$ and regularization
parameter $\varepsilon\geq 0$, with $T_{0,c}=T_c$, and 
\[
c_\Gamma(x,y)
=
-4\|\phi(x)\|_{\mathcal H_{\mathcal X}}^2
   \|\psi(y)\|_{\mathcal H_{\mathcal Y}}^2
-16\left\langle
\Gamma,\phi(x)\otimes\psi(y)
\right\rangle_{\operatorname{HS}} .
\]
For a fixed coupling $\gamma$, the optimal operator is $\Gamma_\gamma = \int \phi(x)\otimes\psi(y)d\gamma(x,y)$ \cite[Theorem 3.3]{houry2026gromov}. Taking
$\phi(x)=U_x^\mu$ and $\psi(y)=V_y^\nu$, identifying
$L^2(\mathcal W_{\mathcal X},\lambda_{\mathcal X})\otimes
L^2(\mathcal W_{\mathcal Y},\lambda_{\mathcal Y})
\cong
\operatorname{HS}\left(
L^2(\mathcal W_{\mathcal Y},\lambda_{\mathcal Y}),
L^2(\mathcal W_{\mathcal X},\lambda_{\mathcal X})
\right)$, and setting $a=8\Gamma$ recovers the linearization
underlying our compact dual formulation \eqref{eq:21dual} with $\Gamma_\gamma$ corresponding to $M_\gamma$. The two formulations use this common variational mechanism in different directions. For our statistical analysis, our explicit squared-$L^2$ representation \eqref{eq:centeredl2embedding} and restriction of the auxiliary parameter to the structured class $\mathcal A$ enable the VC, recentering, and metric-entropy arguments given in Section \ref{sec:pf-upper}. In a computational direction,
\cite{houry2026gromov} leverages the Hilbert--Schmidt formulation to
develop scalable Sinkhorn-based methods for entropic GW, and since their variational identity for $\operatorname{GW}_\varepsilon(\mu,\nu)$ also holds at $\varepsilon = 0$, their formulation provides insights into computational methods for the plug-in estimator studied here.
\end{remark}

\begin{remark}[Extension to even-odd pairs]
\label{rmk:even-odd}
Although we use the representation of the Euclidean distance as a squared $L^2$-distance only for $2$-GW, it applies to general even--odd exponent pairs $(2r,2k+1)$ for $r \geq 1$ and $k \geq 0$. Writing
\begin{align*}
D_{2r,2k+1}^{2r}(\mu,\nu)
&=
\inf_{\gamma\in\Pi(\mu,\nu)}
\iint
\Bigl(
\|U_x^\mu-U_{x'}^\mu\|_{L^2(\mathcal W_{\mathcal X},
\lambda_{\mathcal X})}^{4k+2}
-
\|V_y^\nu-V_{y'}^\nu\|_{L^2(\mathcal W_{\mathcal Y},
\lambda_{\mathcal Y})}^{4k+2}
\Bigr)^{2r}
\,d\gamma(x,y)\,d\gamma(x',y')
\end{align*}
and expanding the integrand, we may obtain finitely many coupling-dependent terms analogous to those found in \cite[Appendix A]{paliy2026empirical}:
\[
    \left\langle M_{\gamma}^{a,b,m,\ell},M_{\gamma}^{a',b',m,\ell} \right\rangle_{\mathcal{H}_{m,\ell}}, \quad 
    M_{\gamma}^{a,b,m,\ell} := \int \|U_x^\mu\|_{L^2(\mathcal W_{\mathcal X},
\lambda_{\mathcal X})}^{2a}\|V_y^\nu\|_{L^2((\mathcal W_{\mathcal Y})}^{2b} (U_x^\mu)^{\otimes m}\otimes (V_y^\nu)^{\otimes \ell}\,d\gamma(x,y),
\]
where $a,a',b,b',m,\ell$ are nonnegative integers and $\mathcal{H}_{m,\ell} = L^2(\mathcal W_{\mathcal X},\lambda_{\mathcal X})^{\otimes m}\otimes L^2(\mathcal W_{\mathcal Y},\lambda_{\mathcal Y})^{\otimes\ell}$. Linearizing these terms may lead to a dual formulation for these higher-order cases. Establishing analogous sample-complexity bounds, however, would require an appropriate recentering argument and new metric-entropy estimates, which we leave for future work.
\end{remark}

%% file: sections/proofs-main-results.tex
\section{Proof of Theorem \ref{thm:main-informal}}\label{sec:proofs-main-results}\label{pf:sample-theorem}
We now prove the sample-complexity rate for the empirical plug-in estimator stated in Theorem~\ref{thm:main-informal}.
\subsection{Proof of upper bound} \label{sec:pf-upper} We first prove the upper bound \eqref{eq:upper-bound} using the duality established in Section~\ref{sec:duality}. For convenience, for $n\geq2$, let
$\psi_d(n)
:=
n^{-\frac{2}{d\vee 4}}
(\log n)^{\mathbf 1_{\{d=4\}}}$. We may assume without loss of generality that $d_x\le d_y$, as the other
case follows by exchanging the roles of the two marginals. Assuming also that $\mu$ and $\nu$ are centered by isometric invariance, take
$R\geq  \operatorname{diam}\mathcal{X} \vee \operatorname{diam}\mathcal{Y}$ so
that $\mathcal{X} \subset B^{d_x}_R(0)$ and
$\mathcal{Y} \subset B^{d_y}_{R}(0)$. We use this same $R$
in the construction of our dual form \eqref{eq:21dual}.

\subsubsection*{Step 1: Expansion via duality}
Since \eqref{eq:21dual} does not require the centering assumption, we can
apply it to both $(\mu,\nu)$ and $(\widehat\mu_n,\widehat\nu_n)$ and use
the generic bound
$|\inf_{z}F(z)-\inf_{z}G(z)|\le\sup_{z}|F(z)-G(z)|$ to split the empirical
estimation error into two terms as
\begin{equation}\label{eq:main-split}
\E\bigl|D_2^2(\widehat\mu_n,\widehat\nu_n)-D_2^2(\mu,\nu)\bigr|
\le
\underbrace{
\E|E_{\widehat\mu_n,\widehat\nu_n}-E_{\mu,\nu}|
}_{=:G_1}+
\underbrace{
\E \sup_{a\in\cA}
\left|T_{c_{a,\widehat\mu_n,\widehat\nu_n}}(\widehat\mu_n,\widehat\nu_n)
-T_{c_{a,\mu,\nu}}(\mu,\nu)\right|
}_{=:G_2}.
\end{equation}
The marginal term $G_1$ can be controlled by a standard $V$-statistic
concentration argument.

\begin{lemma}
\label{lem:v-statistics}
It holds that
$\E|E_{\widehat\mu_n,\widehat\nu_n}-E_{\mu,\nu}|
\lesssim_{R} n^{-1/2}$.
\end{lemma}
The proof is given in Appendix~\ref{pf:v-statistics}.

\subsubsection*{Step 2: Decomposing $G_2$}
Using the triangle inequality after inserting and subtracting
$T_{c_{a,\mu,\nu}}(\widehat\mu_n,\widehat\nu_n)$, followed by
$|T_c(\alpha,\beta)-T_{c'}(\alpha,\beta)|\le\|c-c'\|_\infty$ for bounded
costs $c,c'$ on $\mathcal{X}\times\mathcal{Y}$ and any
$(\alpha,\beta)\in\mathcal{P}(\mathcal{X})\times\mathcal{P}(\mathcal{Y})$,
$G_2$ splits as
\begin{align}\label{eq:main-split2}
  G_2
  &\le
  \E\sup_{a\in\cA}
    \bigl|T_{c_{a,\widehat\mu_n,\widehat\nu_n}}(\widehat\mu_n,\widehat\nu_n)
    -T_{c_{a,\mu,\nu}}(\widehat\mu_n,\widehat\nu_n)\bigr|
  +
  \E\sup_{a\in\cA}
    \bigl|T_{c_{a,\mu,\nu}}(\widehat\mu_n,\widehat\nu_n)
    -T_{c_{a,\mu,\nu}}(\mu,\nu)\bigr|
  \notag\\
  &\le
  \underbrace{\E\sup_{a\in\cA}
    \|c_{a,\widehat\mu_n,\widehat\nu_n}-c_{a,\mu,\nu}\|_\infty}_{=:L_1}
  +
  \underbrace{\E\sup_{a\in\cA}
    \bigl|T_{c_{a,\mu,\nu}}(\widehat\mu_n,\widehat\nu_n)
    -T_{c_{a,\mu,\nu}}(\mu,\nu)\bigr|}_{=:L_2}.
\end{align}

The term $L_1$ quantifies the stability of the cost in its marginal
arguments. For $\alpha,\alpha'\in\mathcal{P}(\mathcal{X})$ and
$\beta,\beta'\in\mathcal{P}(\mathcal{Y})$, define
$\Delta_{\mathcal X}(\alpha,\alpha')
:= \sup_{\xi\in\mathcal W_{\mathcal X}}\abs{\alpha(H_\xi)-\alpha'(H_\xi)}$
and
$\Delta_{\mathcal Y}(\beta,\beta')
:= \sup_{\zeta\in\mathcal W_{\mathcal Y}}\abs{\beta(G_\zeta)-\beta'(G_\zeta)}$,
which measure the discrepancy of the empirical measures from the true ones
over the affine half-spaces. We first observe the following stability
result.
\begin{lemma}[Stability of the cost in the marginals]\label{lem:cost-stability}
For all $\alpha,\alpha'\in\mathcal{P}(\mathcal{X})$ and
$\beta,\beta'\in\mathcal{P}(\mathcal{Y})$,
\begin{equation}\label{eq:cost-stability}
  \sup_{a\in\cA}
  \|c_{a,\alpha,\beta}-c_{a,\alpha',\beta'}\|_{\infty,\mathcal X\times\mathcal Y}
  \le 40\kappa_{\mathcal X}\kappa_{\mathcal Y}
  \bigl(\Delta_{\mathcal X}(\alpha,\alpha')+\Delta_{\mathcal Y}(\beta,\beta')\bigr).
\end{equation}
\end{lemma}

\noindent The proof is given in Appendix~\ref{pf:cost-stability} and is
mostly arithmetic based on the definition of the cost. Since each $H_\xi$
is the restriction of an affine half-space to $\mathcal{X}$, the class
$\{H_\xi:\xi\in\mathcal W_{\mathcal X}\}$ has VC dimension at most
$d_x+1$, and the standard VC maximal inequality (see, e.g.,
\cite[Proposition 3.6.6, Theorem 3.6.9, Corollary 3.5.8]{GineNickl2016}) gives
$\E\Delta_{\mathcal X}(\widehat\mu_n,\mu)\lesssim_{d_x}n^{-1/2}$ and
$\E\Delta_{\mathcal Y}(\widehat\nu_n,\nu)\lesssim_{d_y}n^{-1/2}$.
Combining this with \eqref{eq:cost-stability} (with $\alpha=\widehat\mu_n$,
$\alpha'=\mu$, $\beta=\widehat\nu_n$, and $\beta'=\nu$) and taking expectations, we obtain the parametric rate
$L_1 \lesssim_{R,d_x,d_y} n^{-1/2}$.

\subsubsection*{Step 3: Estimating $L_2$}
The estimation of $L_2$ is more involved, and we split the argument into
three substeps.

\emph{Step 3.1: Marginal recentering of $c_{a_\eta,\mu,\nu}$.}
Fix $a_{\eta} \in \mathcal{A}$. We first construct a correction function
$\phi_{\mu,\eta}:\mathcal{X}\to\R$, depending only on $\mu$ and $\eta$,
such that both $\phi_{\mu,\eta}$ and the recentered cost
$x\mapsto c^\circ_{a_\eta,\mu,\nu}(x,y)
:= c_{a_\eta,\mu,\nu}(x,y)+\phi_{\mu,\eta}(x)$ for each $y\in \mathcal{Y}$
are restrictions to $\mathcal{X}$ of concave functions on
$Q_\mathcal{X}:= B_{R}^{d_x}(0)$. To this end, recall the alternative form
\eqref{eq:cost-alternate} of $c_{a_\eta,\mu,\nu}$ and the associated
measures $\theta_{\eta,y}$ and $\Lambda_{\mu,\eta}$ in Proposition \ref{lem:signed-mixture}. Since
$-\int \|z-x\|\,d\rho(z)$ is a concave function of $x$ whenever the measure
$\rho$ is positive, a natural choice is to subtract from
$\int \|z-x\|\,d\theta_{\eta,y}(z)$ in \eqref{eq:signed-mixture} the
$\Lambda_{\mu,\eta}$-analogue. Namely, we take the correction function
$\phi_{\mu,\eta}:\mathcal{X}\to\R$ given by
\begin{equation}\label{eq:phi}
  \phi_{\mu,\eta}(x):=-\int \|z-x\|\,d\Lambda_{\mu,\eta}(z).
\end{equation}

Since $|\theta_{\eta,y}|\le \Lambda_{\mu,\eta}$, we may fix a Radon--Nikodym derivative
$p_{\eta,y}:=\frac{d\theta_{\eta,y}}{d\Lambda_{\mu,\eta}}$ with
$|p_{\eta,y}|\le 1$ $\Lambda_{\mu,\eta}$-a.e. Substituting
\eqref{eq:signed-mixture} and \eqref{eq:phi} into the definition of
$c^\circ_{a_\eta,\mu,\nu}$ and writing
$q_{\eta,y}:=1-p_{\eta,y}\in[0,2]$, the recentered cost admits the
expression
\begin{equation}\label{eq:positive-mixture}
  c^\circ_{a_\eta,\mu,\nu}(x,y)
  = M_\eta(y)+\int q_{\eta,y}(z)\bigl(-\|z-x\|\bigr)\,d\Lambda_{\mu,\eta}(z).
\end{equation}
By definition, the formulas \eqref{eq:phi} and \eqref{eq:positive-mixture}
(with $y\in \mathcal{Y}$ fixed) extend verbatim to $Q_\mathcal{X}$, and we
denote the corresponding extensions by $\overline{\phi}_{\mu,\eta}$ and
$\overline{c}_{a_\eta,\mu,\nu}(\cdot, y)$. Since $\Lambda_{\mu,\eta}$ is positive and
$q_{\eta,y}\ge 0$, both $\overline{\phi}_{\mu,\eta}$ and
$x\mapsto \overline{c}_{a_\eta,\mu,\nu}(x,y)$ are concave on
$Q_\mathcal{X}$.

\emph{Step 3.2: Expressing $L_2$ using $c^\circ_{a_\eta,\mu,\nu}$ and $\phi_{\mu,\eta}$.}
Since $\phi_{\mu,\eta}$ is a correction in the $x$-variable, for any marginal pair $(\alpha,\beta)$,
\begin{equation}\label{eq:OT-change}
    T_{c_{a_\eta,\mu,\nu}}(\alpha,\beta)
= T_{c^\circ_{a_\eta,\mu,\nu}}(\alpha,\beta) - \int \phi_{\mu,\eta}\,d\alpha.
\end{equation}
Define the recentered cost class and the correction class
\[\mathcal C^\circ_{\mu,\nu}
:= \bigl\{c^\circ_{a_\eta,\mu,\nu}: \|\eta\|_{\TV}\le 2\bigr\},
\qquad
\Phi_\mu := \bigl\{\phi_{\mu,\eta}: \|\eta\|_{\TV}\le 2\bigr\}.\]
For a continuous cost $c$ on $\mathcal{X}\times\mathcal{Y}$ and
$f\in C(\mathcal{X})$, recall the $c$-transform
$f^{c}(y)=\inf_{x\in\mathcal X}\{c(x,y)-f(x)\}$, and let $\mathcal F_c$
denote the class of $c$-concave potentials $f$ (i.e., $f=(f^c)^c$)
normalized so that $\sup_{\mathcal X}f=0$. The dual potential classes
associated with the recentered costs are
\begin{equation}\label{eq:potential-classes}
  \mathcal F^\circ_{\mu,\nu}
   := \bigcup_{c^\circ\in\mathcal C^\circ_{\mu,\nu}}\mathcal F_{c^\circ}
   \subset C(\mathcal X),
  \qquad
  \mathcal G^\circ_{\mu,\nu}
   := \bigcup_{c^\circ\in\mathcal C^\circ_{\mu,\nu}}
      \bigl\{f^{c^\circ}: f\in\mathcal F_{c^\circ}\bigr\}
   \subset C(\mathcal Y).
\end{equation}
Noting that an optimal dual potential exists for each
$c^\circ \in \mathcal C^\circ_{\mu,\nu}$
by \cite[Theorem~5.10(iii)]{villani2009optimal}, Kantorovich duality implies that the suprema of 
empirical errors for OT costs $T_{c^\circ}$ can be bounded as
\begin{equation}\label{eq:OT-dual-ep-bound}
\sup_{c^\circ\in \mathcal C^\circ_{\mu,\nu}}
\left|T_{c^\circ}(\widehat\mu_n,\widehat\nu_n)
- T_{c^\circ}(\mu,\nu)\right|
\leq
\sup_{f\in\mathcal F^\circ_{\mu,\nu}}
\Bigl|\int f\,d(\widehat{\mu}_n-\mu)\Bigr|
+\sup_{g\in\mathcal G^\circ_{\mu,\nu}}
\Bigl|\int g\,d(\widehat{\nu}_n-\nu)\Bigr|.
\end{equation}
Combining \eqref{eq:OT-change} and \eqref{eq:OT-dual-ep-bound}, and using the triangle inequality, $L_2$ is bounded by
the suprema of empirical processes indexed by $\Phi_{\mu}$,
$\mathcal F^\circ_{\mu,\nu}$, and
$\mathcal G^\circ_{\mu,\nu}$:
\begin{equation}\label{eq:G22-three-terms}
 L_2
 \leq
 \underbrace{
 \E\sup_{\phi\in\Phi_\mu}\Bigl|\int \phi\,d(\widehat{\mu}_n-\mu)\Bigr|
 }_{=:S_1}
+\underbrace{
 \E\sup_{f\in\mathcal F^\circ_{\mu,\nu}}
 \Bigl|\int f\,d(\widehat{\mu}_n-\mu)\Bigr|
 }_{=:S_2}
+\underbrace{
 \E\sup_{g\in\mathcal G^\circ_{\mu,\nu}}
 \Bigl|\int g\,d(\widehat{\nu}_n-\nu)\Bigr|
 }_{=:S_3}.
\end{equation}

\emph{Step 3.3: Empirical process bounds.}
It remains to bound each term on the right-hand side of
\eqref{eq:G22-three-terms}. Suppose that $\mathcal F$ is a uniformly
bounded function class whose metric entropy satisfies
$\log \mathcal N(\varepsilon,\mathcal F,\|\cdot\|_\infty)
\leq K\varepsilon^{-l/2}$ for $0<\varepsilon\leq 1$, where $K,l>0$.
Assuming without loss of generality that $0\in\mathcal F$, upon normalizing
the uniform bound to one, we obtain from Dudley's entropy integral
(cf. \cite[Theorem~16]{vonLuxburgBousquet2004})
\begin{align*}
\E\sup_{f\in\mathcal{F}}
\left|\int fd(\widehat{\mu}_n-\mu)\right|
&\lesssim
\inf_{0<\alpha\leq 1}
\left\{
\alpha
+
\frac{1}{\sqrt n}
\int_\alpha^1
\sqrt{\log \mathcal N(\varepsilon,\mathcal{F},\|\cdot\|_\infty)}
d\varepsilon
\right\} \\
&\lesssim
\inf_{0<\alpha\leq 1}
\left\{
\alpha
+
\frac{1}{\sqrt n}
\int_\alpha^1
\varepsilon^{-l/4}
d\varepsilon
\right\} \\
&\lesssim
\psi_l(n),
\end{align*}
where the hidden constants depend on $l$, $K$, and the uniform bound on
$\mathcal F$.\footnote{The last inequality follows upon noticing that the
integral remains finite as $\alpha\downarrow0$ when $l<4$. When $l\geq4$,
differentiating in $\alpha$ and setting the derivative to zero gives the
balancing choice $\alpha\asymp n^{-2/l}$, which yields $\psi_l(n)$.}
Our argument thus reduces to
bounding the metric entropy of $\Phi_\mu$,
$\mathcal F^\circ_{\mu,\nu}$, and
$\mathcal G^\circ_{\mu,\nu}$.

In what follows, we will repeatedly use the fact that if $T$ is an
$L$-Lipschitz map between normed spaces, then
$\mathcal N(\varepsilon,T(\mathcal F),\|\cdot\|)\le
\mathcal N(\varepsilon/L,\mathcal F,\|\cdot\|)$ for any class
$\mathcal F$, so Lipschitz transforms preserve metric entropy exponents. We first collect the regularity of the following
function classes.

\begin{lemma}[Regularity of costs and corrections]
\label{prop:ambient-regularity}
Let $\mathcal C_{\mu,\nu}=\{c_{a_\eta,\mu,\nu}:\|\eta\|_{\TV}\le 2\}$
be the class of original costs, let $\overline\Phi_\mu$ be the class of
extended functions on $Q_{\mathcal X}$ for $\Phi_\mu$, and let
$\overline{\mathcal C}_{\mu,\nu}$ be the class of extended functions
$x\mapsto \overline{c}_{a_\eta,\mu,\nu}(x,y)$ of
$\mathcal C^\circ_{\mu,\nu}$ on $Q_{\mathcal X}$. Then
\begin{enumerate}
\item $\mathcal C_{\mu,\nu}$ and $\mathcal C^\circ_{\mu,\nu}$ are
uniformly bounded and uniformly Lipschitz in each variable;
\item $\overline\Phi_\mu$ and $\overline{\mathcal C}_{\mu,\nu}$ are
uniformly bounded, uniformly Lipschitz, and concave on $Q_{\mathcal X}$.
\end{enumerate}
\end{lemma}
The proof is given in Appendix~\ref{pf:ambient-regularity}. We do not
state the uniform constants explicitly here, but they depend on $R$,
$\kappa_\mathcal{X}$, and $\kappa_\mathcal{Y}$ (hence on $d_x$ and $d_y$) and are recorded at
the end of the proof of each part.

\emph{Step 3.3.1: Bounding $S_1$.}
Corollary~2.7.10 in \cite{VanDerVaartWellner1996} states that, given $L>0$
and a compact convex subset $Q\subset\mathbb{R}^d$, the class
$\mathcal{F} := \{ f:Q\to[0,1] : f \text{ convex and }
\|f\|_{\Lip}\le L \}$ satisfies
$\log \mathcal N(\varepsilon,\mathcal{F},\|\cdot\|_\infty)
\le K(1+L)^{d/2}\varepsilon^{-d/2}$, where $K$ depends only on $d$ and
$Q$. The same bound holds for concave functions by replacing $f$ with
$1-f$. More generally, if $\mathcal{F}$ is a class of concave functions on
$Q$ satisfying $\|f\|_{\infty,Q}\le B$ and
$\|f\|_{\Lip}\le L$, then applying the preceding metric entropy bound to
the normalized functions $\widetilde{f} := (f+B)/(2B)$ gives
\begin{align}\label{eq:metric entropy conc}
  \log \mathcal N\bigl(\varepsilon,\mathcal{F},\|\cdot\|_\infty\bigr)
    \le \log \mathcal N\Bigl(\tfrac{\varepsilon}{2B},
    \{\widetilde{f}\}_{f\in \mathcal{F}},\|\cdot\|_\infty\Bigr)
    \lesssim_{d,Q,B,L} \varepsilon^{-d/2}.
\end{align}
This bound directly applies to $\overline{\Phi}_\mu$ by
Lemma~\ref{prop:ambient-regularity}, with $B$ and $L$ depending on $R$ and
$d_y$, and $Q$ depending on $R$ and $d_x$. Hence
$\log \mathcal N(\varepsilon, \overline{\Phi}_\mu, \|\cdot\|_\infty)
\lesssim_{R,d_x,d_y} \varepsilon^{-d_x/2}$. Since restriction from
$Q_\mathcal{X}$ to $\mathcal X$ is $1$-Lipschitz in the sup norm, the
metric entropy of $\Phi_\mu$ is bounded by that of its extended class
$\overline{\Phi}_\mu$. Hence
$\log \mathcal N(\varepsilon, \Phi_\mu, \|\cdot\|_\infty)
\lesssim_{R,d_x,d_y} \varepsilon^{-d_x/2}$, and consequently
$S_1 \lesssim_{R,d_x,d_y} \psi_{d_x}(n)$ by Dudley's entropy integral formula.

\emph{Step 3.3.2: Bounding $S_2$.}
For estimating the metric entropy of $\mathcal F^\circ_{\mu,\nu}$,
recall that each $f \in \mathcal F_{c^\circ}$ is the
$c^\circ$-conjugate of $f^{c^\circ}$.
With this characterization, we see that the potential class
$\mathcal F^\circ_{\mu,\nu}$ inherits the regularity and concavity of
the cost family $\overline{\mathcal C}_{\mu,\nu}$ established in
Lemma~\ref{prop:ambient-regularity}.
\begin{lemma}[Regularity of dual potentials] \label{lem:reg of F}
For each $f\in \mathcal{F}_{c^\circ} \subset \mathcal F^\circ_{\mu,\nu}$, define its extension
$\overline{f}:Q_\mathcal{X} \to \R$ via the
$\overline{c}$-conjugate transform
$\overline{f}(x) = \inf_{y\in \mathcal{Y}}
\{\overline{c}(x, y) - f^{c^\circ}(y)\}$. Then the extended function
class $\overline{\mathcal F}_{\mu,\nu}
:= \{\overline{f}: f\in \mathcal F_{c^\circ},    \ c^\circ \in \mathcal C^\circ_{\mu,\nu}\}$ is uniformly
bounded, uniformly Lipschitz, and concave on $Q_\mathcal{X}$.
\end{lemma}
The proof, along with the corresponding uniform constants, is given in Appendix~\ref{pf:dual-potential-regularity}.
Applying the estimate \eqref{eq:metric entropy conc} to
$\overline{\mathcal F}_{\mu,\nu}$ and invoking the $1$-Lipschitzness of
restriction from $Q_\mathcal{X}$ to $\mathcal{X}$ as in the previous step,
we arrive at
$\log \mathcal N(\varepsilon, \mathcal F^\circ_{\mu,\nu},
\|\cdot\|_\infty) \lesssim_{R,d_x,d_y} \varepsilon^{-d_{x}/2}$ and hence $S_2 \lesssim_{R,d_x,d_y} \psi_{d_x}(n)$.

\emph{Step 3.3.3: Bounding $S_3$.}
It remains to bound the metric entropy of
$\mathcal G^\circ_{\mu,\nu}$. For
$c^\circ_1, c^\circ_2 \in \mathcal C^\circ_{\mu,\nu}$ and
$f_1,f_2 \in \mathcal F^\circ_{\mu,\nu}$, by the generic bound used
in Step~1, we have
$\|f_1^{c^\circ_1} - f_2^{c^\circ_2}\|_\infty
\leq \|f_1-f_2\|_\infty + \|c^\circ_1-c^\circ_2\|_\infty$, which gives
\begin{equation}\label{eq:metric-entropy-tildeG}
\log \mathcal N\bigl(\varepsilon,\mathcal G^\circ_{\mu,\nu},
\|\cdot\|_\infty\bigr)
\le
\log \mathcal N\bigl(\varepsilon/2,\mathcal F^\circ_{\mu,\nu},
\|\cdot\|_\infty\bigr)
+ \log \mathcal N\bigl(\varepsilon/2,\mathcal C^\circ_{\mu,\nu},
\|\cdot\|_\infty\bigr).
\end{equation}
From the previous step, the first term is
$\lesssim_{R,d_x,d_y}\varepsilon^{-d_x/2}$. By the definition of
$\mathcal C^\circ_{\mu,\nu}$, the second term is similarly bounded by two terms
as
\begin{align} \label{eq:metric-entropy-tildeC}
    \log \mathcal N\bigl(\varepsilon/2,\mathcal C^\circ_{\mu,\nu},
\|\cdot\|_\infty\bigr)
\leq \log \mathcal N\bigl(\varepsilon/4,\Phi_{\mu},
\|\cdot\|_\infty\bigr)
+\log \mathcal N\bigl(\varepsilon/4,\mathcal C_{\mu,\nu},
\|\cdot\|_\infty\bigr),
\end{align}
the first of which is $\lesssim_{R,d_x,d_y}\varepsilon^{-d_x/2}$ by
Step~3.3.1. For estimating the metric entropy of $\mathcal C_{\mu,\nu}$,
applying the Cauchy--Schwarz inequality and the bounds
$\|U_x^\mu\|_{L^2(\mathcal W_\mathcal X,\lambda_\mathcal X)} \leq \sqrt{\kappa_{\mathcal{X}}}$ and
$\|V_y^\nu\|_{L^2(\mathcal W_\mathcal Y,\lambda_\mathcal Y)} \leq \sqrt{\kappa_{\mathcal{Y}}}$
(Corollary~\ref{cor:crofton-centered}) to the original form
\eqref{eq:cost-original} of the cost gives that $c_{a,\mu,\nu}$ is
Lipschitz in $a\in \mathcal{A}$ with respect to the
$L^2(\mathcal{W},\lambda)$-norm:
\begin{align}
\bigl\|c_{a,\mu,\nu}-c_{b,\mu,\nu}\bigr\|_{\infty, \mathcal{X}\times \mathcal{Y}}
&\leq
2\|a-b\|_{L^2(\mathcal W,\lambda)}
\sup_{x\in\mathcal X}\|U_x^\mu\|_{L^2(\mathcal W_{\mathcal X},\lambda_{\mathcal X})}
\sup_{y\in\mathcal Y}\|V_y^\nu\|_{L^2(\mathcal W_{\mathcal Y},\lambda_{\mathcal Y})} \notag \\
&\leq 2\sqrt{\kappa_{\mathcal X}\kappa_{\mathcal Y}}\|a-b\|_{L^2(\mathcal{W}, \lambda)}. \label{eq:lipschitz-of-c}
\end{align}
Recall that $\mathcal{A}$ consists of elements of the form
$8\sqrt{\kappa_{\mathcal{X}}\kappa_{\mathcal{Y}}}\int h_{u,v}\,d\eta(u,v)$.
By definition, the Bochner integral $\int h_{u,v}\,d\eta(u,v)$ is
constructed as the $L^2$-limit of integrals of simple functions
approximating $(u,v)\mapsto h_{u,v}$ in $L^1(\mathcal X\times\mathcal Y,|\eta|;L^2(\mathcal W,\lambda))$, and these integrals take
the form $\sum_{j=1}^{m}a_j g_j$ with $\sum_{j=1}^m|a_j|\le2$ and
$g_j\in L^2(\mathcal{W},\lambda)$. As demonstrated in the proof of
Lemma~\ref{lem:convex-hull} below, the elements $g_j$ can be taken from
$\{h_{u,v}\}_{(u,v)\in\mathcal{X}\times\mathcal{Y}}$, so that $\mathcal{A}$
is contained in the scaled $L^2$-closed absolutely convex hull
$16\sqrt{\kappa_\mathcal{X}\kappa_\mathcal{Y}}
\overline{\operatorname{aco}}\{h_{u,v}\}_{(u,v)\in\mathcal{X}\times\mathcal{Y}}$.
Combining this with the $1/2$-H\"older continuity of
$(u,v)\mapsto h_{u,v}$ (Lemma~\ref{lem:holder-smooth}) and a standard
entropy bound for closed convex hulls
(\cite[Theorem~2.6.9]{VanDerVaartWellner1996}), the metric entropy of
$\mathcal{A}$ can be bounded as follows.
\begin{lemma}\label{lem:convex-hull}
For $0<\varepsilon\le1$,
\[\log\mathcal N\bigl(\varepsilon,\mathcal{A},
\|\cdot\|_{L^2(\mathcal{W},\lambda)}\bigr)
\lesssim_{R,d_x,d_y} \varepsilon^{-q_{d_x,d_y}},\qquad
q_{d_x,d_y}=\frac{2(d_x+d_y)}{d_x+d_y+1}.\]
\end{lemma}
The proof is given in Appendix~\ref{pf:convex-hull}. It follows by
\eqref{eq:lipschitz-of-c} that
$\log \mathcal N\bigl(\varepsilon/4,\mathcal C_{\mu,\nu},
\|\cdot\|_\infty\bigr) \lesssim_{R,d_x,d_y} \varepsilon^{-q_{d_x,d_y}}$.
Combining this with the bounds for the metric entropies of $\mathcal F^\circ_{\mu,\nu}$ and $\Phi_{\mu}$ in
\eqref{eq:metric-entropy-tildeG} and \eqref{eq:metric-entropy-tildeC}
gives
\[\log \mathcal N\bigl(\varepsilon,\mathcal G^\circ_{\mu,\nu},
\|\cdot\|_\infty\bigr) \lesssim_{R,d_x,d_y} \varepsilon^{-r},
\qquad r := \max\{q_{d_x,d_y},\, d_x/2\}.\]
If $r = d_x/2$, then $S_3 \lesssim_{R,d_x,d_y} \psi_{d_x}(n)$. Otherwise,
$S_3 \lesssim_{R,d_x,d_y} \psi_{2q_{d_x,d_y}}(n)$, but since we always have
$2q_{d_x,d_y} < 4$, the rate is parametric:
$\psi_{2q_{d_x,d_y}}(n) = n^{-1/2} = \psi_{d_x}(n)$. Therefore, in either
case we always have $S_3 \lesssim_{R,d_x,d_y} \psi_{d_x}(n)$.

\subsubsection*{Step 4: Combine}
By Steps 3.3.1--3.3.3, we deduce that $L_2$ has the dominant rate
$\lesssim_{R,d_x,d_y} \psi_{d_x}(n)$. Combining this with $G_1$ and $L_1$, which have the parametric
rate $n^{-1/2}$, we conclude that
\begin{align*}
    \E |D_2^2(\widehat\mu_n, \widehat\nu_n) - D_2^2(\mu, \nu)|
    \leq G_1 + L_1 + L_2
    \lesssim_{R,d_x,d_y} \psi_{d_x}(n).
\end{align*}
\subsection{Proof of lower bound}
Since the parametric lower bound $n^{-1/2}$ for
$d_x\wedge d_y\leq4$ follows from the central limit theorem (for instance,
by considering $\mu=\frac34\delta_0+\frac14\delta_{e_1}$ and
$\nu=\delta_0$, where $e_1$ is the first standard basis vector in
$\R^{d_x}$), it remains
to consider $d_x\wedge d_y>4$. Let $\varrho$ be the uniform distribution
on $B_1^{d_x\wedge d_y}(0)$, and let $\mu$ and $\nu$ be the pushforwards
of $\varrho$ into $\R^{d_x}$ and $\R^{d_y}$ under the canonical embeddings,
respectively. Let $\widehat\varrho_n$ and $\widehat\varrho_n'$ be independent
empirical measures based on $\varrho$, so that $\widehat\mu_n$ and
$\widehat\nu_n$ are their pushforwards under the respective canonical
embeddings. Since the embeddings are isometries, we have
$D_2^2(\mu,\nu)=D_{2,2}^2(\mu,\nu)=0$. Moreover, by isometric invariance,
$D_{2,2}(\widehat\mu_n,\widehat\nu_n)
=D_{2,2}(\widehat\varrho_n,\widehat\varrho_n')$, and equation~(29) in the
proof of the lower bound for $D_{2,2}^2$ in
\cite[Section~5.3.2]{zhang2024gromov} gives
\begin{equation*}
\E D_{2,2}^2(\widehat\mu_n,\widehat\nu_n)
\gtrsim_{d_x,d_y}
\E\lambda_{\min}(\Sigma_{\widehat\varrho_n})
n^{-2/(d_x\wedge d_y)},
\end{equation*}
where $\Sigma_\theta:=\int zz^\top\,d\theta(z)$. The
estimate immediately following equation~(29) shows that
$\E\lambda_{\min}(\Sigma_{\widehat\varrho_n})\geq
\frac12\lambda_{\min}(\Sigma_\varrho)$ for all sufficiently large $n$, where $\lambda_{\min}(\Sigma_\theta)$ is the smallest eigenvalue of $\Sigma_\theta$.
Since $\Sigma_\varrho=I/(d_x\wedge d_y+2)$, we have
$\lambda_{\min}(\Sigma_\varrho)=1/(d_x\wedge d_y+2)$, and therefore
\begin{equation*}
\E D_{2,2}^2(\widehat\mu_n,\widehat\nu_n)
\gtrsim_{d_x,d_y}
n^{-2/(d_x\wedge d_y)}.
\end{equation*}
Since the supports have diameter at most $2$,
\cite[Lemma~9.5(iii)]{sturm2023space} gives
$D_{2,2}\leq4D_2$. Therefore,
\[
\E\left|
D_2^2(\widehat\mu_n,\widehat\nu_n)-D_2^2(\mu,\nu)
\right| =
\E D_2^2(\widehat\mu_n,\widehat\nu_n) \geq
\frac1{16}\E D_{2,2}^2(\widehat\mu_n,\widehat\nu_n) \gtrsim_{d_x,d_y}
n^{-2/(d_x\wedge d_y)}.
\]

%% file: appendix/proofs-crofton-duality.tex
\section{Proofs for Section \ref{sec:crofton-halfspace}}\label{sec:crofton-identities}

\subsection{Proof of Theorem~\ref{theorem:global-crofton}}
\label{pf:global-crofton}

Fix $x,x'\in\R^d$. The case $x=x'$ is trivial, so assume
$x\neq x'$. For each $\theta\in\Sd^{d-1}$ and $t\in\R$, we have
\[
\left(
\one_{\{\theta\cdot x>t\}}
-
\one_{\{\theta\cdot x'>t\}}
\right)^2
=
\one_{\{
\min(\theta\cdot x,\theta\cdot x')
\leq t<
\max(\theta\cdot x,\theta\cdot x')
\}},
\]
and hence
\[
\int_{-\infty}^{\infty}
\left(
\one_{\{\theta\cdot x>t\}}
-
\one_{\{\theta\cdot x'>t\}}
\right)^2
\,dt
=
|\theta\cdot(x-x')|.
\]
Recalling the definition of $\Lambda_d$, it follows that
\[
\int
\left(
\one_{H_\xi}(x)-\one_{H_\xi}(x')
\right)^2
\,d\Lambda_d(\xi)
=
\frac{1}{a_d}
\int
|\theta\cdot(x-x')|
\,d\sigma_{d-1}(\theta).
\]
Choose an orthogonal matrix $Q\in\R^{d\times d}$ such that
$Q(x-x')=\|x-x'\|e_1$, where $e_1\in\R^d$ is the first standard basis
vector. With the change of variables $\omega=Q\theta$ and the
rotational invariance of $\sigma_{d-1}$,
\[
\int
|\theta\cdot(x-x')|
\,d\sigma_{d-1}(\theta)
=
\|x-x'\|
\int
|\omega_1|
\,d\sigma_{d-1}(\omega)
=
a_d\|x-x'\|.
\]
Therefore
\[
\int
\left(
\one_{H_\xi}(x)-\one_{H_\xi}(x')
\right)^2
\,d\Lambda_d(\xi)
=
\|x-x'\|.
\]
\subsection{Proof of Corollary~\ref{cor:global-crofton-centered}}
\label{pf:global-crofton-centered}
By definition,
$\widetilde U_x^\mu(\xi)
=
\int
\bigl(\one_{H_\xi}(x)-\one_{H_\xi}(z)\bigr)
d\mu(z).$
Hence, Jensen's inequality, Fubini and
Theorem~\ref{theorem:global-crofton} together give
\[
\|\widetilde U_x^\mu\|^2_{L^2(\mathcal{W}_d, \Lambda_d)} \leq
\int\!\!\int
\bigl(\one_{H_\xi}(x)-\one_{H_\xi}(z)\bigr)^2
\,d\Lambda_d(\xi)\,d\mu(z) =
\int\|x-z\|\,d\mu(z) <\infty.
\]
Thus $\widetilde U_x^\mu\in L^2(\mathcal W_d,\Lambda_d)$.
\eqref{eq:global-centeredl2embedding} readily follows from \eqref{eq:global-crofton}, and the continuity of $x\mapsto \widetilde U_x^\mu$ immediately follows. Moreover,
\[
\int\|\widetilde U_x^\mu\|^2_{L^2(\mathcal W_d,\Lambda_d)}\,d\mu(x)
\leq
\iint\|x-z\|\,d\mu(x)\,d\mu(z)
<\infty.
\]
Thus the Cauchy--Schwarz inequality implies that $x\mapsto \widetilde U_x^{\mu}$ is Bochner integrable.
Finally, for any $h\in L^2(\mathcal W_d,\Lambda_d)$, Lemma \ref{lem:Bochner} and Fubini give
\begin{align*}
\left\langle
\int \widetilde U_x^\mu\,d\mu(x),h
\right\rangle_{L^2(\mathcal W_d,\Lambda_d)}
&=
\int\langle \widetilde U_x^\mu,h\rangle\,d\mu(x)\\
&=
\iint
\widetilde U_x^\mu(\xi)h(\xi)
\,d\Lambda_d(\xi)\,d\mu(x) =
\int
\left(
\int \widetilde U_x^\mu(\xi)\,d\mu(x)
\right)
h(\xi)\,d\Lambda_d(\xi).
\end{align*}
Since $\int \widetilde U_x^\mu(\xi)\,d\mu(x)=0$ for every $\xi\in\mathcal W_d$,
$\left\langle
\int \widetilde U_x^\mu\,d\mu(x),h
\right\rangle_{L^2(\mathcal W_d,\Lambda_d)}
=
0.$ Therefore
$\int \widetilde U_x^\mu\,d\mu(x)=0$ in
$L^2(\mathcal W_d,\Lambda_d)$.

\subsection{Proof of Lemma \ref{lem:crofton-overlap}}
\label{pf:crofton-overlap}
Fix $x,x'\in \mathcal{X} \subset B^d_R(0)$. For each $\theta \in \Sd^{d-1}$,
both $\theta\cdot x > t$ and $\theta\cdot x'> t$ hold if and only if
$\min(\theta\cdot x, \theta\cdot x') > t$. Since
$\theta\cdot x, \theta\cdot x' \in [-R,R]$,
\[
    \frac{1}{2R}\int_{-R}^R
    \mathbf 1_{\{\theta\cdot x>t\}}
    \mathbf 1_{\{\theta\cdot x'>t\}}
    \, dt
    = \frac{R+\min(\theta\cdot x,\theta\cdot x')}{2R}
    = \frac{1}{2} + \frac{\min(\theta\cdot x,\theta\cdot x')}{2R}.
\]
Using $\min(a,b)=\frac{a+b-|a-b|}{2}$, we obtain
\begin{align*}
\int
    \min(\theta\cdot x,\theta\cdot x')
    \, d\sigma_{d-1}(\theta)
&=
\frac12
\int
    (\theta\cdot x+\theta\cdot x'
    -|\theta\cdot(x-x')|)
    \, d\sigma_{d-1}(\theta) \\
&=
\frac12 (x+x')\cdot
\int
    \theta
    \, d\sigma_{d-1}(\theta)
-
\frac12
\int
    |\theta\cdot(x-x')|
    \, d\sigma_{d-1}(\theta) \\
&=
-\frac12
\int
    |\theta\cdot(x-x')|
    \, d\sigma_{d-1}(\theta),
\end{align*}
where we used $\int\theta \, d\sigma_{d-1}(\theta)=0$. Using the rotational
invariance of $\sigma_{d-1}$ as in the proof of
Theorem~\ref{theorem:global-crofton}, we obtain
\[
    \int
        \min(\theta\cdot x,\theta\cdot x')
    \,d\sigma_{d-1}(\theta)
    =
    -\frac{a_d}{2}\|x-x'\|.
\]
Therefore
\[
    \int
    \mathbf 1_{H_\xi}(x)\mathbf 1_{H_\xi}(x')
    \,d\lambda_{\mathcal{X}}(\xi)
    =
    \frac12-\frac{a_d}{4R}\|x-x'\|
    =
    \frac12-\frac{\|x-x'\|}{2\kappa_{\mathcal{X}}}.
\]

%% file: appendix/proofs-crofton-halfspace.tex
\subsection{Proof of Lemma \ref{lem:linear-distance}}
\label{pf:linear-distance}

\begin{proof}
Expanding the square in the definition of $\|U_x^\mu\|^2_{L^2(\mathcal W_\mathcal X,\lambda_\mathcal X)}$ gives
\begin{align*}
\|U_x^\mu\|^2_{L^2(\mathcal W_\mathcal X,\lambda_\mathcal X)}
&=\kappa_{\mathcal X}\left[\int\1_{H_\xi}(x)d\lambda_{\mathcal X}(\xi)
 -2\int\1_{H_\xi}(x)\mu(H_\xi)d\lambda_{\mathcal X}(\xi)
 +\int \mu(H_\xi)^2d\lambda_{\mathcal X}(\xi)\right].
\end{align*}
The first integral is $\frac12$. Indeed, writing $\xi=(\theta,t)$, for
fixed $\theta\in \Sd^{d-1}$ we have
\[
\int_{-R}^{R}\mathbf 1_{\{\theta\cdot x>t\}}
dt
=R+\theta\cdot x
\]
because $\theta\cdot x\in[-R,R]$. Integrating over all $\theta \in \Sd^{d-1}$ gives
\begin{equation}\label{eq:1/2eq}
\int\mathbf 1_{H_\xi}(x)\,d\lambda_{\mathcal X}(\xi)
=
\frac12+
\frac{1}{2R}\int\theta\cdot x\,d\sigma_{d-1}(\theta)
=
\frac12,
\end{equation}
where the last equality follows by symmetry. Next, since
$\mu(H_\xi)=\int \mathbf 1_{H_\xi}(z)\,d\mu(z),$
Fubini gives
\begin{align*}
\int\mathbf 1_{H_\xi}(x)\mu(H_\xi)\,d\lambda_{\mathcal X}(\xi)
&=
\iint 
\mathbf 1_{H_\xi}(x)\mathbf 1_{H_\xi}(z)
\,d\mu(z)\,d\lambda_{\mathcal X}(\xi) \\
&=
\int 
\left(
\int
\mathbf 1_{H_\xi}(x)\mathbf 1_{H_\xi}(z)
\,d\lambda_{\mathcal X}(\xi)
\right)
d\mu(z) \\
&= \int 
\left(
\frac12-\frac{\|z-x\|}{2\kappa_{\mathcal X}}
\right)
d\mu(z),
\end{align*}
where for the last equality we used \eqref{eq:crofton-overlap}.
Substituting this into the expansion of $\|U_x^\mu\|^2_{L^2(\mathcal W_\mathcal X,\lambda_\mathcal X)}$, we obtain
\begin{align*}
\|U_x^\mu\|^2_{L^2(\mathcal W_\mathcal X,\lambda_\mathcal X)}
&= \kappa_{\mathcal X} \left[
\frac12
-
2\left(
\frac12-\frac{1}{2\kappa_{\mathcal X}}\int \|z-x\|\,d\mu(z)
\right)
+
\int\mu(H_\xi)^2\,d\lambda_{\mathcal X}(\xi)
\right] \\
&=
\int \|z-x\|\,d\mu(z)
+
\kappa_{\mathcal X}
\left(
\int\mu(H_\xi)^2\,d\lambda_{\mathcal X}(\xi)-\frac12
\right).
\end{align*}

Thus
\[
\|U_x^\mu\|^2_{L^2(\mathcal W_\mathcal X,\lambda_\mathcal X)}
=b_\mu+\int \|z-x\|\,d\mu(z),
\qquad
b_\mu
:=
\kappa_{\mathcal X}
\left(
\int\mu(H_\xi)^2\,d\lambda_{\mathcal X}(\xi)-\frac12
\right).
\]

Finally, for $x,x'\in \mathcal X$, the reverse triangle inequality gives 
\[
\left|
\|U_x^\mu\|^2_{L^2(\mathcal W_\mathcal X,\lambda_\mathcal X)}
-
\|U_{x'}^\mu\|^2_{L^2(\mathcal W_\mathcal X,\lambda_\mathcal X)}
\right| \le
\int \bigl|\|z-x\|-\|z-x'\|\bigr|\,d\mu(z)\le \|x-x'\|.
\]
\end{proof}

\subsection{Proof of Lemma \ref{lem:linear-distance-y}}
\label{pf:linear-distance-y}
Substituting the definition
$U_x^\mu(\xi)=\sqrt{\kappa_{\mathcal X}}(\1_{H_\xi}(x)-\mu(H_\xi))$ into that
of $\alpha_u^\mu$ gives
\[
\alpha_u^\mu(x)
=\int \1_{H_\xi}(u)\1_{H_\xi}(x)\,d\lambda_{\mathcal X}(\xi)
-\int \1_{H_\xi}(u)\mu(H_\xi)\,d\lambda_{\mathcal X}(\xi).
\]
Since $|\1_{H_\xi}(x)-\mu(H_\xi)|\le 1$, the identity \eqref{eq:1/2eq} yields
\[
|\alpha_u^\mu(x)|
\le\int \1_{H_\xi}(u)\,d\lambda_{\mathcal X}(\xi)
=\frac12.
\]
Moreover, applying \eqref{eq:crofton-overlap} to the first integral above,
\[
\alpha_u^\mu(x)
=\frac12-\frac{\|u-x\|}{2\kappa_{\mathcal X}}
-\int \1_{H_\xi}(u)\mu(H_\xi)\,d\lambda_{\mathcal X}(\xi)
=A_\mu(u)-\frac{\|u-x\|}{2\kappa_{\mathcal X}},
\]
which is \eqref{eq:linear-alpha}. Finally, for $x,x'\in\mathcal X$, the
reverse triangle inequality gives
\[
|\alpha_u^\mu(x)-\alpha_u^\mu(x')|
=\frac{\bigl|\|u-x\|-\|u-x'\|\bigr|}{2\kappa_{\mathcal X}}
\le\frac{\|x-x'\|}{2\kappa_{\mathcal X}},
\]
so $\alpha_u^\mu$ is $1/(2\kappa_{\mathcal X})$-Lipschitz in $x$.

%% file: appendix/dual.tex
\section{Proofs for Section \ref{sec:duality}}
\label{appendix:proofs-duality}

\subsection{Proof of Theorem~\ref{thm:global-duality}}
\label{pf:global-duality-proof}

\emph{Step 1: Expanding $D_2^2$.}
We directly expand the integrand in $D_2^2$ to decompose it as
$D_2^2(\mu,\nu)
=
B_{\mu,\nu}
+
\inf_{\gamma\in\Pi(\mu,\nu)}C_\gamma$, where
\begin{align*}
B_{\mu,\nu}
&=
\iint \|x-x'\|^2\,d\mu(x)\,d\mu(x')
+
\iint \|y-y'\|^2\,d\nu(y)\,d\nu(y'),\\
C_\gamma
&=
-2\iint
\|x-x'\|\|y-y'\|
\,d\gamma(x,y)\,d\gamma(x',y').
\end{align*}
The term $B_{\mu,\nu}$ depends only on the marginals, while $C_\gamma$
is the coupling-dependent term.

\emph{Step 2: Expanding $C_\gamma$.}
We now rewrite $C_\gamma$ using the squared $L^2$ representation of
Euclidean distance from
Corollary~\ref{cor:global-crofton-centered} and expand it further.
For notational convenience, let
$L^2_x=L^2(\mathcal W_{d_x},\Lambda_{d_x})$ and
$L^2_y=L^2(\mathcal W_{d_y},\Lambda_{d_y})$.
Let $(X,Y)\sim\gamma$, let $(X',Y')$ be an independent copy of $(X,Y)$,
and write $\E$ for expectation with respect to the joint law
$\gamma\otimes\gamma$ of $((X,Y),(X',Y'))$. Recall the global centered
embeddings $\widetilde U^\mu:\R^{d_x}\to L^2_x$ and
$\widetilde V^\nu:\R^{d_y}\to L^2_y$ from \eqref{eq:global-centered-crofton}. From the proof of Corollary~\ref{cor:global-crofton-centered},
$\|\widetilde U_x^\mu\|_{L^2_x}^2
\leq\int\|x-z\|\,d\mu(z)$ and
$\|\widetilde V_y^\nu\|_{L^2_y}^2
\leq\int\|y-w\|\,d\nu(w)$. Hence
$\int\|\widetilde U_x^\mu\|_{L^2_x}^4\,d\mu(x)<\infty$ and
$\int\|\widetilde V_y^\nu\|_{L^2_y}^4\,d\nu(y)<\infty$, and substituting
$\|x-x'\|=\|\widetilde U_x^\mu-\widetilde U_{x'}^\mu\|_{L^2_x}^2$ and
$\|y-y'\|=\|\widetilde V_y^\nu-\widetilde V_{y'}^\nu\|_{L^2_y}^2$ into $C_\gamma$ and
expanding both squared norms gives
\begin{align*}
-2\iint
&\|x-x'\|\|y-y'\|
\,d\gamma(x,y)\,d\gamma(x',y')\\
&=
-2\E\left[
\|\widetilde U_X^\mu-\widetilde U_{X'}^\mu\|_{L^2_x}^2
\|\widetilde V_Y^\nu-\widetilde V_{Y'}^\nu\|_{L^2_y}^2
\right]\\
&=
-2\E\left[
\left(
\|\widetilde U_X^\mu\|_{L^2_x}^2
+
\|\widetilde U_{X'}^\mu\|_{L^2_x}^2
-
2\langle \widetilde U_X^\mu,\widetilde U_{X'}^\mu\rangle_{L^2_x}
\right)
\left(
\|\widetilde V_Y^\nu\|_{L^2_y}^2
+
\|\widetilde V_{Y'}^\nu\|_{L^2_y}^2
-
2\langle \widetilde V_Y^\nu,\widetilde V_{Y'}^\nu\rangle_{L^2_y}
\right)
\right]\\
&=
-4\E_{X\sim\mu}\left[\|\widetilde U_X^\mu\|_{L^2_x}^2\right]
\E_{Y\sim\nu}\left[\|\widetilde V_Y^\nu\|_{L^2_y}^2\right]
-4\E\left[
\|\widetilde U_X^\mu\|_{L^2_x}^2
\|\widetilde V_Y^\nu\|_{L^2_y}^2
\right]\\
&\qquad
-8\E\left[
\langle \widetilde U_X^\mu,\widetilde U_{X'}^\mu\rangle_{L^2_x}
\langle \widetilde V_Y^\nu,\widetilde V_{Y'}^\nu\rangle_{L^2_y}
\right],
\end{align*}
where all terms obtained from the expansion are finite by the above bounds on $\|\widetilde{U}_x^\mu\|^2_{L_x^2}$ and $\|\widetilde{V}_y^\nu\|^2_{L_y^2}$. Furthermore, the cross terms containing exactly one inner product vanish by
the tower law since $\E \widetilde U_X^\mu=\E \widetilde V_Y^\nu=0$. The first term depends
only on the marginals and will be absorbed into $B_{\mu,\nu}$ in Step~4,
and the second is linear in $\gamma$. Only the third term is quadratic
in $\gamma$, and we treat it next.

\emph{Step 3: Linearizing the quadratic term.}
For each $(x,y)\in\R^{d_x}\times\R^{d_y}$, define the tensor
$\widetilde U_x^\mu\otimes \widetilde V_y^\nu\in L^2(\widetilde{\mathcal W},\widetilde\Lambda)$ by
$(\widetilde U_x^\mu\otimes \widetilde V_y^\nu)(\xi,\zeta)
=
\widetilde U_x^\mu(\xi)\widetilde V_y^\nu(\zeta)$.
The mapping $(x,y)\mapsto \widetilde U_x^\mu\otimes \widetilde V_y^\nu$ is continuous since
\begin{equation} \label{eq:tensor-bound}
\begin{aligned}
\|\widetilde U_x^\mu\otimes \widetilde V_y^\nu
-\widetilde U_{x'}^\mu\otimes \widetilde V_{y'}^\nu\|_{L^2(\widetilde{\mathcal W},\widetilde\Lambda)}
&\leq
\|\widetilde U_x^\mu-\widetilde U_{x'}^\mu\|_{L^2_x}\|\widetilde V_y^\nu\|_{L^2_y}+
\|\widetilde U_{x'}^\mu\|_{L^2_x}
\|\widetilde V_y^\nu-\widetilde V_{y'}^\nu\|_{L^2_y},
\end{aligned}
\end{equation}
whose right-hand side converges to $0$ as $(x',y')\to(x,y)$ by the
squared $L^2$ representation. Moreover, by the Cauchy--Schwarz inequality,
\begin{align*}
\int
\|\widetilde U_x^\mu\otimes \widetilde V_y^\nu\|_{L^2(\widetilde{\mathcal W},\widetilde\Lambda)}
\,d\gamma(x,y)
&=
\int
\|\widetilde U_x^\mu\|_{L^2_x}\|\widetilde V_y^\nu\|_{L^2_y}
\,d\gamma(x,y)\\
&\leq
\sqrt{m_\mu m_\nu}
<\infty.
\end{align*}
Therefore the Bochner integral $\widetilde M_\gamma
=\int
\widetilde U_x^\mu\otimes \widetilde V_y^\nu\,d\gamma(x,y)$ is well-defined. By Lemma~\ref{lem:Bochner} and Fubini,
\begin{align*}
-8\E\left[
\langle \widetilde U_X^\mu,\widetilde U_{X'}^\mu\rangle_{L^2_x}
\langle \widetilde V_Y^\nu,\widetilde V_{Y'}^\nu\rangle_{L^2_y}
\right]
&=
-8\iint
\langle \widetilde U_x^\mu,\widetilde U_{x'}^\mu\rangle_{L^2_x}
\langle \widetilde V_y^\nu,\widetilde V_{y'}^\nu\rangle_{L^2_y}
\,d\gamma(x,y)\,d\gamma(x',y')\\
&=
-8\left\langle
\int \widetilde U_x^\mu\otimes \widetilde V_y^\nu\,d\gamma(x,y),
\int \widetilde U_{x'}^\mu\otimes \widetilde V_{y'}^\nu\,d\gamma(x',y')
\right\rangle_{L^2(\widetilde{\mathcal W},\widetilde\Lambda)}\\
&=
-8\|\widetilde M_\gamma\|_{L^2(\widetilde{\mathcal W},\widetilde\Lambda)}^2.
\end{align*}
This squared norm can be linearized by completing the square:
\begin{equation}\label{eq:frencherel}
-8\|\widetilde M_\gamma\|_{L^2(\widetilde{\mathcal W},\widetilde\Lambda)}^2
=
\inf_{a\in L^2(\widetilde{\mathcal W},\widetilde\Lambda)}
\left\{
\frac18\|a\|_{L^2(\widetilde{\mathcal W},\widetilde\Lambda)}^2
-
2\langle a,\widetilde M_\gamma\rangle_{L^2(\widetilde{\mathcal W},\widetilde\Lambda)}
\right\},
\end{equation}
with the unique minimizer $a^*=8\widetilde M_\gamma$. 
Since $a^*\in\mathcal B_{\mu,\nu}$, the infimum in \eqref{eq:frencherel} may be
restricted to $\mathcal B_{\mu,\nu}$.

\emph{Step 4: Combine.}
Since $\E \widetilde U_X^\mu=0$ and $\E \widetilde V_Y^\nu=0$, we have
\begin{align*}
\iint\|x-x'\|\,d\mu(x)\,d\mu(x')
&=
2\E_{X\sim\mu}\left[\|\widetilde U_X^\mu\|_{L^2_x}^2\right], \quad
\iint\|y-y'\|\,d\nu(y)\,d\nu(y')
&=
2\E_{Y\sim\nu}\left[\|\widetilde V_Y^\nu\|_{L^2_y}^2\right].
\end{align*}
Hence, the marginal terms combine as
$B_{\mu,\nu}
-
4\E_{X\sim\mu}\|\widetilde U_X^\mu\|_{L^2_x}^2
\E_{Y\sim\nu}\|\widetilde V_Y^\nu\|_{L^2_y}^2
=
E_{\mu,\nu}$.

Using Lemma~\ref{lem:Bochner} and Fubini,
\begin{equation}\label{eq:global-pairing}
\langle a,\widetilde M_\gamma\rangle_{L^2(\widetilde{\mathcal W},\widetilde\Lambda)}
=
\int_{\R^{d_x}\times\R^{d_y}}
\int_{\widetilde{\mathcal W}}
a(\xi,\zeta)\widetilde U_x^\mu(\xi)\widetilde V_y^\nu(\zeta)
\,d\widetilde\Lambda(\xi,\zeta)\,d\gamma(x,y).
\end{equation}
The fourth-moment bounds from Step~2 and Cauchy--Schwarz imply that
$\widetilde c_{a,\mu,\nu}$ is integrable with respect to every
$\gamma\in\Pi(\mu,\nu)$. Substituting \eqref{eq:global-pairing} into
\eqref{eq:frencherel}, restricting the infimum to
$\mathcal B_{\mu,\nu}$, and collecting all terms gives
\begin{align*}
D_2^2(\mu,\nu)
&=
E_{\mu,\nu}
+
\inf_{\gamma\in\Pi(\mu,\nu)}
\inf_{a\in\mathcal B_{\mu,\nu}}
\left\{
\frac18\|a\|_{L^2(\widetilde{\mathcal W},\widetilde\Lambda)}^2
+
\int_{\R^{d_x}\times\R^{d_y}}
\widetilde c_{a,\mu,\nu}(x,y)\,d\gamma(x,y)
\right\}\\
&=
E_{\mu,\nu}
+
\inf_{a\in\mathcal B_{\mu,\nu}}
\left\{
\frac18\|a\|_{L^2(\widetilde{\mathcal W},\widetilde\Lambda)}^2
+
T_{\widetilde c_{a,\mu,\nu}}(\mu,\nu)
\right\}.
\end{align*}
\subsection{Proof of Corollary~\ref{cor:compact-duality}}
\label{pf:compact-duality-proof}
Since Steps~1--3 follow verbatim with the compact centered embeddings $U^\mu$ and $V^\nu$ from
\eqref{eq:centered-crofton} in place of $\widetilde U^\mu$ and $\widetilde V^\nu$, it suffices to verify that $a^*=8M_\gamma$ belongs to
$\mathcal A$.

We first verify that $\mathcal A$ is well-defined, as the following lemma gives
the continuity of $(u,v)\mapsto h_{u,v}$ (boundedness follows from
$\|h_{u,v}\|_{L^2(\mathcal W,\lambda)}\leq1$).

\begin{lemma}\label{lem:holder-smooth}
For the mapping $(u,v)\mapsto h_{u,v}$ given by
$h_{u,v}(\xi,\zeta)
=
\mathbf 1_{H_\xi}(u)\mathbf 1_{G_\zeta}(v)$,
\begin{equation*}
\|h_{u,v}-h_{u',v'}\|_{L^2(\mathcal W,\lambda)}^2
\leq
\frac{\|u-u'\|}{\kappa_{\mathcal X}}
+
\frac{\|v-v'\|}{\kappa_{\mathcal Y}}.
\end{equation*}
Consequently, $h$ is $1/2$-H\"older with respect to the product metric
$d((u,v),(u',v'))=\|u-u'\|+\|v-v'\|$ on
$\mathcal X\times\mathcal Y$.
\end{lemma}
\begin{proof}
    For $(u,v),(u',v')\in \mathcal X\times \mathcal Y$ and $(\xi,\zeta)\in \mathcal W_{\mathcal X}\times \mathcal W_{\mathcal Y}$, we have
\begin{align*}
\bigl|h_{u,v}(\xi,\zeta)-h_{u',v'}(\xi,\zeta)\bigr|^2
&=
\bigl|
\mathbf 1_{H_\xi}(u)\mathbf 1_{G_\zeta}(v)
-
\mathbf 1_{H_\xi}(u')\mathbf 1_{G_\zeta}(v')
\bigr|^2 \\
&\leq
\mathbf 1\!\left\{
\mathbf 1_{H_\xi}(u)\mathbf 1_{G_\zeta}(v)
\neq
\mathbf 1_{H_\xi}(u')\mathbf 1_{G_\zeta}(v')
\right\} \\
&\leq
\mathbf 1\!\left\{
\mathbf 1_{H_\xi}(u)\neq \mathbf 1_{H_\xi}(u')
\right\}
+
\mathbf 1\!\left\{
\mathbf 1_{G_\zeta}(v)\neq \mathbf 1_{G_\zeta}(v')
\right\} \\
&=
\bigl|\mathbf 1_{H_\xi}(u)-\mathbf 1_{H_\xi}(u')\bigr|
+
\bigl|\mathbf 1_{G_\zeta}(v)-\mathbf 1_{G_\zeta}(v')\bigr| \\
&= (\mathbf 1_{H_\xi}(u)-\mathbf 1_{H_\xi}(u'))^2
+
\bigl(\mathbf 1_{G_\zeta}(v)-\mathbf 1_{G_\zeta}(v')\bigr)^2.
\end{align*}
Integrating over $(\xi, \zeta) \in \mathcal W_{\mathcal X} \times \mathcal W_{\mathcal Y}$ with respect to $\lambda$ and using Corollary~\ref{thm:euclidean-crofton} gives
\begin{equation*}
\|h_{u,v}-h_{u',v'}\|_{L^2(\mathcal{W},\lambda)}^2
\le \frac{\|u-u'\|}{\kappa_{\mathcal X}}+\frac{\|v-v'\|}{\kappa_{\mathcal Y}}.
\end{equation*}
Taking square roots, we deduce that the map $(u,v)\mapsto h_{u,v}$ is $1/2$-H\"older. 
\end{proof}
To verify that
$a^*\in\mathcal A$, we use the following representation of $M_\gamma$.

\begin{lemma}\label{lem:Mgamma representation}
It holds that
\begin{equation}\label{eq:pointwise}
M_\gamma
=
\sqrt{\kappa_{\mathcal X}\kappa_{\mathcal Y}}
\int_{\mathcal X\times\mathcal Y}
h_{x,y}\,d(\gamma-\mu\otimes\nu)(x,y),
\end{equation}
where the integral is interpreted as the Bochner integral of
$(x,y)\mapsto h_{x,y}$.
\end{lemma}
\begin{proof}
    As in the proof of Corollary~\ref{cor:global-crofton-centered},
Lemma~\ref{lem:Bochner} and Fubini identify each Bochner integral with
the $L^2$-equivalence class of its pointwise integral. Since $\gamma$ has
marginals $\mu$ and $\nu$, for each $(\xi,\zeta)\in\mathcal W$ the
pointwise representative of $M_\gamma$ satisfies
\begin{align*}
M_\gamma(\xi,\zeta)
&=
\int_{\mathcal X\times\mathcal Y}
U_x^\mu(\xi)V_y^\nu(\zeta)\,d\gamma(x,y)\\
&=
\sqrt{\kappa_\mathcal X\kappa_\mathcal Y}
\int_{\mathcal X\times\mathcal Y}
\left(\one_{H_\xi}(x)-\mu(H_\xi)\right)
\left(\one_{G_\zeta}(y)-\nu(G_\zeta)\right)
\,d\gamma(x,y)\\
&=
\sqrt{\kappa_\mathcal X\kappa_\mathcal Y}
\left(
\gamma(H_\xi\times G_\zeta)
-\mu(H_\xi)\nu(G_\zeta)
\right).
\end{align*}
On the other hand, the pointwise representative of the right-hand side of
\eqref{eq:pointwise} is
\begin{align*}
&\sqrt{\kappa_\mathcal X\kappa_\mathcal Y}
\int_{\mathcal X\times\mathcal Y}
h_{x,y}(\xi,\zeta)\,d(\gamma-\mu\otimes\nu)(x,y)\\
&\qquad=
\sqrt{\kappa_\mathcal X\kappa_\mathcal Y}
\left(
\int_{\mathcal X\times\mathcal Y}
\one_{H_\xi}(x)\one_{G_\zeta}(y)\,d\gamma(x,y)
-
\int_{\mathcal X\times\mathcal Y}
\one_{H_\xi}(x)\one_{G_\zeta}(y)\,d(\mu\otimes\nu)(x,y)
\right)\\
&\qquad=
\sqrt{\kappa_\mathcal X\kappa_\mathcal Y}
\left(
\gamma(H_\xi\times G_\zeta)
-\mu(H_\xi)\nu(G_\zeta)
\right).
\end{align*}
Therefore they coincide in $L^2(\mathcal W,\lambda)$.
\end{proof}
Since
$\|\gamma-\mu\otimes\nu\|_{\TV}\leq2$ by the triangle inequality,
the definition of $\mathcal A$ gives $a^*=8M_\gamma\in\mathcal A$.
Hence, the infimum in \eqref{eq:frencherel} may be
restricted to $\mathcal A$ instead, and repeating Step~4 yields \eqref{eq:21dual}.

%% file: appendix/proofs-duality-cost.tex
\section{Proofs for Section~\ref{sec:cost-structure}}\label{sec:duality-cost-proofs}

\subsection{Proof of Proposition \ref{lem:signed-mixture}}
\label{pf:signed-mixture}
Substituting the identities in Lemmas~\ref{lem:linear-distance} and~\ref{lem:linear-distance-y} into \eqref{eq:cost-alternate} gives
\begin{align*}
c_{a_\eta,\mu,\nu}(x,y)
&=
-4\|V_y^\nu\|^2_{L^2(\mathcal W_\mathcal Y,\lambda_\mathcal Y)}
\left(
b_\mu+\int \|z-x\|\,d\mu(z)
\right)  \\
&\qquad \qquad
-
16\kappa_{\mathcal X}\kappa_{\mathcal Y}
\int
\left(
A_\mu(u)-\frac{\|u-x\|}{2\kappa_{\mathcal X}}
\right)
\beta_v^\nu(y)\,d\eta(u,v) \\
&=
-4\|V_y^\nu\|^2_{L^2(\mathcal W_\mathcal Y,\lambda_\mathcal Y)}b_\mu
-
16\kappa_{\mathcal X}\kappa_{\mathcal Y}
\int A_\mu(u)\beta_v^\nu(y)\,d\eta(u,v)
\\
&\qquad \qquad
-4\|V_y^\nu\|^2_{L^2(\mathcal W_\mathcal Y,\lambda_\mathcal Y)}
\int \|z-x\|\,d\mu(z)
+
8\kappa_{\mathcal Y}
\int\|u-x\|\beta_v^\nu(y)\,d\eta(u,v).
\end{align*}
The first two terms are independent of $x$, so for each fixed $y\in \mathcal Y$, set
\[
M_\eta(y)
:=
-4\|V_y^\nu\|^2_{L^2(\mathcal W_\mathcal Y,\lambda_\mathcal Y)}b_\mu
-
16\kappa_{\mathcal X}\kappa_{\mathcal Y}
\int A_\mu(u)\beta_v^\nu(y)\,d\eta(u,v).
\]
Define a signed Borel measure
$\rho_{\eta,y} := (\pi_1)_\sharp\bigl(\beta_v^\nu(y)\eta(du,dv)\bigr)$
on $\mathcal X$, which is finite since $|\beta_v^\nu(y)|\le \frac12$ by Lemma \ref{lem:linear-distance-y}. Then for every Borel set $E\subset \mathcal X$,
\[
\rho_{\eta,y}(E)
=
\left[(\pi_1)_\sharp\bigl(\beta_v^\nu(y)\eta\bigr)\right](E)
=
\bigl(\beta_v^\nu(y)\eta\bigr)\bigl(\pi_1^{-1}(E)\bigr)
=
\int_{E\times \mathcal Y} \beta_v^\nu(y)\,d\eta(u,v).
\]
Then
$\int\|u-x\|\beta_v^\nu(y)\,d\eta(u,v) = \int \|z-x\|\,d\rho_{\eta,y}(z).$
Therefore
\[
c_{a_\eta,\mu,\nu}(x,y)
=
M_\eta(y)
+
\int \|z-x\|\,d\theta_{\eta,y}(z), \qquad 
\theta_{\eta,y}
:=
-4\|V_y^\nu\|^2_{L^2(\mathcal W_\mathcal Y,\lambda_\mathcal Y)}\mu
+
8\kappa_{\mathcal Y}\rho_{\eta,y}.
\]
In other words,
\[
\theta_{\eta,y}
=
-4\|V_y^\nu\|^2_{L^2(\mathcal W_\mathcal Y,\lambda_\mathcal Y)}\mu
+
8\kappa_{\mathcal Y}(\pi_1)_\sharp\bigl(\beta_v^\nu(y)\eta(du,dv)\bigr).
\]
This proves \eqref{eq:signed-mixture}. It remains to prove $|\theta_{\eta,y}|\le \Lambda_{\mu,\eta}$. We have $|\beta_v^\nu(y)|\le \frac12$, and as measures on $\mathcal X$,
\[
|\rho_{\eta,y}|
\le
(\pi_1)_\sharp\bigl(|\beta_v^\nu(y)|\,|\eta|\bigr)
\le
\frac12(\pi_1)_\sharp|\eta|.
\]
Hence
\begin{align*}
|\theta_{\eta,y}|
&\le
4\|V_y^\nu\|^2_{L^2(\mathcal W_\mathcal Y,\lambda_\mathcal Y)}\mu
+8\kappa_{\mathcal Y}|\rho_{\eta,y}|
\leq 
4\kappa_{\mathcal Y}\mu
+4\kappa_{\mathcal Y}(\pi_1)_\sharp|\eta|.
\end{align*}
Thus $|\theta_{\eta,y}|\leq \Lambda_{\mu,\eta}$, where $\Lambda_{\mu,\eta}
:=
4\kappa_{\mathcal Y}\mu
+
4\kappa_{\mathcal Y}(\pi_1)_\sharp|\eta|$, and 
\[
\Lambda_{\mu,\eta}(\mathcal X)
=
4\kappa_{\mathcal Y}\mu(\mathcal X)
+
4\kappa_{\mathcal Y}|\eta|(\mathcal X\times \mathcal Y)
=
4\kappa_{\mathcal Y}
+
4\kappa_{\mathcal Y}\|\eta\|_{\TV}.
\]
Since $\|\eta\|_{\TV}\le 2$, we conclude that $\Lambda_{\mu,\eta}(\mathcal X) \leq 12\kappa_{\mathcal Y}$.

%% file: appendix/proofs-sample-complexity.tex
\section{Proofs for Section~\ref{pf:sample-theorem}}\label{sec:sample-proofs}
\subsection{Proof of Lemma \ref{lem:v-statistics}}\label{pf:v-statistics}

Write $\widehat\mu_n=\frac1n\sum_{i=1}^n\delta_{X_i}$ and
$\widehat\nu_n=\frac1n\sum_{i=1}^n\delta_{Y_i}$. For simplicity, write
$E_{\mu,\nu}=Q_\mu+Q_\nu-D_\mu D_\nu$, where
\[
Q_\mu:=\iint\|x-x'\|^2\,d\mu(x)d\mu(x'),
\qquad
D_\mu:=\iint\|x-x'\|\,d\mu(x)d\mu(x'),
\]
and analogously for $Q_\nu$ and $D_\nu$. Since
$Q_{\widehat\mu_n}=\frac1{n^2}\sum_{i,j=1}^n \|X_i-X_j\|^2$
and $D_{\widehat\mu_n}=\frac1{n^2}\sum_{i,j=1}^n \|X_i-X_j\|$, they are
$V$-statistics with kernels $h_Q(x,x'):=\|x-x'\|^2$ and
$h_D(x,x'):=\|x-x'\|$. Since $\mathcal X\subset B^{d_x}_{R}(0)$, we have
$\|h_Q\|_\infty\le 4R^2$ and $\|h_D\|_\infty\le 2R$.

For a kernel $h$, let $V_n^h:=n^{-2}\sum_{i,j=1}^n h(Z_i,Z_j)$, where
$Z_1,\ldots,Z_n$ are i.i.d.\ with law $\rho$, and suppose
$\|h\|_\infty\le B$. Changing one coordinate changes at most $2n-1$
summands, and hence changes $V_n^h$ by at most $4B/n$. By McDiarmid's
inequality \cite{mcdiarmid1989bounded},
$\E|V_n^h-\E V_n^h|\lesssim Bn^{-1/2}$. Also,
\[
\E V_n^h
=
\frac{n(n-1)}{n^2}\iint h\,d\rho d\rho
+
\frac1n\int h(z,z)\,d\rho(z),
\]
so $|\E V_n^h-\iint h\,d\rho d\rho|\le 2B/n$. Therefore
$\E\left|V_n^h-\iint h\,d\rho d\rho\right|\lesssim Bn^{-1/2}$.
Applying this to $h_Q$ and $h_D$ gives
$\E|Q_{\widehat\mu_n}-Q_\mu|\lesssim R^2n^{-1/2}$ and
$\E|D_{\widehat\mu_n}-D_\mu|\lesssim Rn^{-1/2}$, and analogously
$\E|Q_{\widehat\nu_n}-Q_\nu|\lesssim R^2n^{-1/2}$ and
$\E|D_{\widehat\nu_n}-D_\nu|\lesssim Rn^{-1/2}$.

Finally,
\[
|E_{\widehat\mu_n,\widehat\nu_n}-E_{\mu,\nu}|
\le
|Q_{\widehat\mu_n}-Q_\mu|
+
|Q_{\widehat\nu_n}-Q_\nu|
+
|D_{\widehat\mu_n}D_{\widehat\nu_n}-D_\mu D_\nu|.
\]
Since $0\le D_{\widehat\mu_n},D_\mu\le 2R$ and
$0\le D_{\widehat\nu_n},D_\nu\le 2R$,
\[
|D_{\widehat\mu_n}D_{\widehat\nu_n}-D_\mu D_\nu|
\le
2R|D_{\widehat\mu_n}-D_\mu|
+
2R|D_{\widehat\nu_n}-D_\nu|.
\]
Taking expectations and combining these estimates yields $\E|E_{\widehat\mu_n,\widehat\nu_n}-E_{\mu,\nu}|
\lesssim_R n^{-1/2}.$

\subsection{Proof of Lemma \ref{lem:cost-stability}}
\label{pf:cost-stability}
Fix $\alpha,\alpha'\in\mathcal P(\mathcal X)$ and
$\beta,\beta'\in\mathcal P(\mathcal Y)$. Recall the original
form~\eqref{eq:cost-original} of the cost $c_{a,\mu,\nu}$. Since
$|(u-p)^2-(u-p')^2|\le 2|p-p'|$ for $u\in\{0,1\}$ and $p,p'\in[0,1]$,
recalling the definitions of $U_x^\alpha$ and $V_y^\beta$, and writing
$L^2_\mathcal X:=L^2(\mathcal W_\mathcal X,\lambda_\mathcal X)$ and
$L^2_\mathcal Y:=L^2(\mathcal W_\mathcal Y,\lambda_\mathcal Y)$, for any
$x\in \mathcal{X}$ and $y \in \mathcal{Y}$ we have
\[
  \left|\|U_x^\alpha\|^2_{L^2_\mathcal X}
  - \|U_x^{\alpha'}\|^2_{L^2_\mathcal X}
  \right| \le 2\kappa_\mathcal{X}\Delta_\mathcal{X}(\alpha,\alpha'),
  \qquad
  \left|\|V_y^\beta\|^2_{L^2_\mathcal Y}
  - \|V_y^{\beta'}\|^2_{L^2_\mathcal Y}
  \right| \le 2\kappa_\mathcal{Y}\Delta_\mathcal{Y}(\beta,\beta').
\]
Since
$0\le \|U_x^\alpha\|_{L^2_\mathcal X}^2,\|U_x^{\alpha'}\|_{L^2_\mathcal X}^2\le\kappa_\mathcal{X}$
and
$0\le \|V_y^\beta\|_{L^2_\mathcal Y}^2,\|V_y^{\beta'}\|_{L^2_\mathcal Y}^2\le\kappa_\mathcal{Y}$
by Corollary~\ref{cor:crofton-centered},

\[
\begin{aligned}
4\left|\|U_x^\alpha\|^2_{L^2_\mathcal X}\|V_y^\beta\|^2_{L^2_\mathcal Y}
-\|U_x^{\alpha'}\|^2_{L^2_\mathcal X}\|V_y^{\beta'}\|^2_{L^2_\mathcal Y}\right|
&\le
4\|V_y^\beta\|^2_{L^2_\mathcal Y}
\left|\|U_x^\alpha\|^2_{L^2_\mathcal X}-\|U_x^{\alpha'}\|^2_{L^2_\mathcal X}\right| + 
4\|U_x^{\alpha'}\|^2_{L^2_\mathcal X}
\left|\|V_y^\beta\|^2_{L^2_\mathcal Y}-\|V_y^{\beta'}\|^2_{L^2_\mathcal Y}\right|\\
&\le
8\kappa_{\mathcal X}\kappa_{\mathcal Y}
\bigl(\Delta_{\mathcal X}(\alpha,\alpha')
+\Delta_{\mathcal Y}(\beta,\beta')\bigr).
\end{aligned}
\]

Applying the bound \eqref{eq:tensor-bound} together with
Corollary~\ref{cor:crofton-centered}, we obtain
\begin{align*}
\|U_x^\alpha\otimes V_y^\beta-U_x^{\alpha'}\otimes V_y^{\beta'}\|_{L^2(\mathcal W,\lambda)}
&\le
\|U_x^\alpha-U_x^{\alpha'}\|_{L^2_\mathcal X}
\|V_y^\beta\|_{L^2_\mathcal Y}
+
\|U_x^{\alpha'}\|_{L^2_\mathcal X}
\|V_y^\beta-V_y^{\beta'}\|_{L^2_\mathcal Y}\\
&\le
\sqrt{\kappa_{\mathcal X}\kappa_{\mathcal Y}}
\bigl(\Delta_{\mathcal X}(\alpha,\alpha')
+\Delta_{\mathcal Y}(\beta,\beta')\bigr).
\end{align*}
Combining this with
$\|a\|_{L^2(\mathcal W,\lambda)}\le 16\sqrt{\kappa_{\mathcal X}\kappa_{\mathcal Y}}$
and the Cauchy--Schwarz inequality gives
\[
2\left|\int_{\mathcal W}a(\xi,\zeta)
\bigl(U_x^\alpha(\xi)V_y^\beta(\zeta)-U_x^{\alpha'}(\xi)V_y^{\beta'}(\zeta)\bigr)
\,d\lambda(\xi,\zeta)\right|
\le
32\kappa_{\mathcal X}\kappa_{\mathcal Y}
\bigl(\Delta_{\mathcal X}(\alpha,\alpha')+\Delta_{\mathcal Y}(\beta,\beta')\bigr).
\]
Combining these two bounds gives~\eqref{eq:cost-stability}.

\subsection{Proof of Lemma \ref{prop:ambient-regularity}}
\label{pf:ambient-regularity}

\subsubsection*{Proof for $\mathcal C_{\mu,\nu}$}
We first prove that $\mathcal C_{\mu,\nu}$ is uniformly bounded and
uniformly Lipschitz. Recall the alternative form~\eqref{eq:cost-alternate}
for the cost
\[
  c_{a_\eta, \mu, \nu}(x,y)= -4\|U_x^\mu\|^2_{L^2(\mathcal{W}_\mathcal{X}, \lambda_\mathcal{X})}\|V_y^\nu\|^2_{L^2(\mathcal{W}_\mathcal{Y}, \lambda_\mathcal{Y})}
  - 16\kappa_\mathcal{X}\kappa_\mathcal{Y} \int \alpha_u^\mu(x)\beta_v^\nu(y) \, d\eta(u,v).
\]
By Corollary~\ref{cor:crofton-centered} and
Lemma~\ref{lem:linear-distance},
$\|U_x^\mu\|^2_{L^2(\mathcal{W}_\mathcal{X}, \lambda_\mathcal{X})}$ and
$\|V_y^\nu\|^2_{L^2(\mathcal{W}_\mathcal{Y}, \lambda_\mathcal{Y})}$ are
uniformly bounded with bounds $\kappa_{\mathcal X}$ and
$\kappa_{\mathcal Y}$, and $1$-Lipschitz in $x$ and $y$, respectively. It
immediately follows that the first term in \eqref{eq:cost-alternate} is
uniformly bounded with bound $4\kappa_\mathcal{X}\kappa_\mathcal{Y}$,
$4\kappa_\mathcal{Y}$-Lipschitz in $x$, and
$4\kappa_\mathcal{X}$-Lipschitz in $y$.

Now since $|\alpha_u^\mu(x)|, |\beta_v^\nu(y)| \leq \frac{1}{2}$ by
Lemma~\ref{lem:linear-distance-y} and $\|\eta\|_{\TV}\leq 2$, the
second term is uniformly bounded with bound
$8\kappa_\mathcal{X}\kappa_\mathcal{Y}$. Moreover, for fixed $y$, the
$\frac{1}{2\kappa_\mathcal{X}}$-Lipschitzness of $\alpha_u^\mu$ gives
\[
\left|\int \alpha_u^\mu(x)\beta_v^\nu(y)\,d\eta(u,v)
-
\int \alpha_u^\mu(x')\beta_v^\nu(y)\,d\eta(u,v)\right|
\le
\frac{\|x-x'\|}{2\kappa_\mathcal{X}}\cdot \frac12 \|\eta\|_{\TV}
\le
\frac{\|x-x'\|}{2\kappa_\mathcal{X}}.
\]
Hence the second term is $8\kappa_\mathcal{Y}$-Lipschitz in $x$.
Analogously, it is $8\kappa_\mathcal{X}$-Lipschitz in $y$. Combining both
terms, we obtain
\[
\|c_{a_\eta,\mu,\nu}\|_\infty\le 12\kappa_\mathcal{X}\kappa_\mathcal{Y},\qquad
\|c_{a_\eta,\mu,\nu}\|_{\Lip,x}\le 12\kappa_\mathcal{Y},\qquad
\|c_{a_\eta,\mu,\nu}\|_{\Lip,y}\le 12\kappa_\mathcal{X}.
\]

\subsubsection*{Proof for $\overline\Phi_{\mu}$}
By
Proposition~\ref{lem:signed-mixture},
$\Lambda_{\mu,\eta}(\mathcal X)\le 12\kappa_{\mathcal Y}$. Note that 
$\overline{\phi}_{\mu,\eta}(w)=
-\int \|z-w\|\,d\Lambda_{\mu,\eta}(z)$, $w\in Q_\mathcal{X}$. For each fixed $z\in \mathcal X$, the map
$w\mapsto -\|z-w\|$ is concave and $1$-Lipschitz on $Q_{\mathcal X}$. Since
$\Lambda_{\mu,\eta}$ is a positive measure, $\overline{\phi}_{\mu,\eta}$ is
concave on $Q_{\mathcal X}$. For $w,w'\in Q_{\mathcal X}$,
\[
|\overline{\phi}_{\mu,\eta}(w)-\overline{\phi}_{\mu,\eta}(w')|
\le
\int \bigl|\|z-w\|-\|z-w'\|\bigr|\,d\Lambda_{\mu,\eta}(z)
\le
12\kappa_{\mathcal Y}\|w-w'\|.
\]
Also, for $z\in \mathcal X$ and $w\in Q_{\mathcal X}$, $\|z-w\|\le 2R$.
Therefore
\[
|\overline{\phi}_{\mu,\eta}(w)|
\le
\int \|z-w\|\,d\Lambda_{\mu,\eta}(z)
\le
2R\,\Lambda_{\mu,\eta}(\mathcal X)
\le
24\kappa_{\mathcal Y}R.
\]
Thus,
\[
\|\overline{\phi}_{\mu,\eta}\|_{\infty,Q_{\mathcal X}}
\le 24\kappa_{\mathcal Y}R,
\qquad
\|\overline{\phi}_{\mu,\eta}\|_{\Lip(Q_{\mathcal X})}
\le 12\kappa_{\mathcal Y}.
\]

\subsubsection*{Proof for $\mathcal C^\circ_{\mu,\nu}$}
By definition, both uniform boundedness and uniform Lipschitzness of
$\mathcal C^\circ_{\mu,\nu}$ immediately follow from the two preceding proofs, and the bounds are given by
\[
\|c^\circ_{a_\eta,\mu,\nu}\|_{\infty}
\le
12\kappa_{\mathcal X}\kappa_{\mathcal Y}
+
24\kappa_{\mathcal Y}R,
\qquad
\|c^\circ_{a_\eta,\mu,\nu}\|_{\Lip,x}
\le
24\kappa_{\mathcal Y},
\qquad
\|c^\circ_{a_\eta,\mu,\nu}\|_{\Lip,y}
\le
12\kappa_{\mathcal X}.
\]

\subsubsection*{Proof for $\overline{\mathcal C}_{\mu, \nu}$}
Fix $c^\circ_{a_\eta,\mu,\nu}\in\mathcal C^\circ_{\mu,\nu}$ and
$y\in \mathcal Y$. Note that 
\[
\overline{c}_{a_\eta,\mu,\nu}(w,y)
=
M_\eta(y)
+
\int q_{\eta,y}(z)(-\|z-w\|)\,d\Lambda_{\mu,\eta}(z),
\qquad w\in Q_{\mathcal X}.
\]
Concavity was shown in Step 3.1 of the proof of Theorem~\ref{thm:main-informal}. For
Lipschitzness, if $w,w'\in Q_{\mathcal X}$, then
\[
\left|\overline{c}_{a_\eta,\mu,\nu}(w,y)
-
\overline{c}_{a_\eta,\mu,\nu}(w',y) \right|
\le
\int q_{\eta,y}(z)\|w-w'\|\,d\Lambda_{\mu,\eta}(z)
\le
24\kappa_{\mathcal Y} \|w-w'\|,
\]
where we used $0\le q_{\eta,y}\le 2$ and
$\Lambda_{\mu,\eta}(\mathcal X)\le 12\kappa_{\mathcal Y}$.

For uniform boundedness, fix $x_0\in \mathcal X$. By the uniform
boundedness of $c^\circ_{a_\eta,\mu,\nu}$ from the last part,
for every $w\in Q_{\mathcal X}$,
\[
|\overline{c}_{a_\eta,\mu,\nu}(w,y)|
\le
|\overline{c}_{a_\eta,\mu,\nu}(x_0,y)|
+
24\kappa_{\mathcal Y}\|w-x_0\|
\le
12\kappa_\mathcal{X}\kappa_\mathcal{Y}
+
24\kappa_\mathcal{Y}R
+
48\kappa_{\mathcal Y}R,
\]
where we have used
$\overline{c}_{a_\eta,\mu,\nu}(x_0,y)
=c^\circ_{a_\eta,\mu,\nu}(x_0,y)$ and $\|w-x_0\|\le 2R$. Thus,
\[
\|\overline{c}_{a_\eta,\mu,\nu}(\cdot, y)\|_{\infty,Q_{\mathcal X}}
\le
12\kappa_\mathcal{X}\kappa_\mathcal{Y}
+
72\kappa_\mathcal{Y}R,
\qquad
\|\overline{c}_{a_\eta,\mu,\nu}(\cdot, y)\|_{\Lip(Q_{\mathcal X})}
\le
24\kappa_{\mathcal Y}.
\]

\subsection{Proof of Lemma~\ref{lem:reg of F}}
\label{pf:dual-potential-regularity}

Fix $f\in\mathcal F^\circ_{\mu,\nu}$, and choose
$c^\circ\in\mathcal C^\circ_{\mu,\nu}$ such that
$f\in\mathcal F_{c^\circ}$ and $\sup_{\mathcal X}f=0$. Set
$B_{R,d_x,d_y}:=12\kappa_{\mathcal X}\kappa_{\mathcal Y}
+72\kappa_{\mathcal Y}R$ and
$L_{R,d_x,d_y}:=24\kappa_{\mathcal Y}$. By Lemma~\ref{prop:ambient-regularity},
$c^\circ$ and $\overline c$ are uniformly bounded by
$B_{R,d_x,d_y}$, while each map
$w\mapsto\overline c(w,y)$ is concave and
$L_{R,d_x,d_y}$-Lipschitz on $Q_{\mathcal X}$. Since $f\le0$, we have
$f^{c^\circ}(y)=\inf_{x\in\mathcal X}
\{c^\circ(x,y)-f(x)\}\ge-B_{R,d_x,d_y}$. Choosing
$x_0\in\mathcal X$ such that $f(x_0)=0$ also gives
$f^{c^\circ}(y)\le c^\circ(x_0,y)\le B_{R,d_x,d_y}$.
Hence $\|f^{c^\circ}\|_{\infty,\mathcal Y}\le B_{R,d_x,d_y}$.
Since $f=(f^{c^\circ})^{c^\circ}$, it follows that
$-2B_{R,d_x,d_y}\le f\le0$.

For each $y\in\mathcal Y$, the map
$w\mapsto\overline c(w,y)-f^{c^\circ}(y)$ is concave and
$L_{R,d_x,d_y}$-Lipschitz. Taking the infimum over $y\in\mathcal Y$
preserves concavity and the common Lipschitz constant. Combined with the
bounds on $\overline c$ and $f^{c^\circ}$, we thus have
\[
\|\overline f\|_{\infty,Q_{\mathcal X}}
\le 2B_{R,d_x,d_y}
=24\kappa_{\mathcal X}\kappa_{\mathcal Y}
+144\kappa_{\mathcal Y}R,
\qquad
\|\overline f\|_{\Lip(Q_{\mathcal X})}
\le L_{R,d_x,d_y}
=24\kappa_{\mathcal Y}.
\]

\subsection{Proof of Lemma~\ref{lem:convex-hull}}
\label{pf:convex-hull}
Throughout, we write $H:=\{h_{u,v}\}_{(u,v)\in\mathcal X\times\mathcal Y}$.

\emph{Step 1: $\mathcal A\subseteq
16\sqrt{\kappa_\mathcal{X}\kappa_\mathcal{Y}}
\overline{\operatorname{aco}}(H)$.}
Fix $\eta$ with $\|\eta\|_{\TV}\le2$ and $\varepsilon>0$. By the
uniform continuity of $(u,v)\mapsto h_{u,v}$ (Lemma~\ref{lem:holder-smooth}), there exists
$\delta>0$ such that
$\|h_{u,v}-h_{u',v'}\|_{L^2(\mathcal W,\lambda)}<\varepsilon$ whenever
$\|(u,v)-(u',v')\|<\delta$. Let $\{z_i\}_{i=1}^K$ be a finite
$\delta$-cover of $\mathcal X\times\mathcal Y$, and define the partition
$\{A_i\}_{i=1}^K$ of $\mathcal X\times\mathcal Y$ by $A_1=(\mathcal X\times \mathcal Y) \cap B^{\circ}(z_1,\delta)$, where $B^{\circ}(z_1,\delta)$ is open, and 
\[A_i=((\mathcal X\times \mathcal Y) \cap B^{\circ}(z_i,\delta))\setminus\bigcup_{j=1}^{i-1}B^{\circ}(z_j,\delta),
\quad i=2,\ldots,K,\]
where we discard any empty ones with no loss of generality. The sets $A_i$
are Borel and disjoint. Consider the simple function
$F_\varepsilon:=\sum_{i=1}^K\mathbf 1_{A_i}h_{z_i}$, whose Bochner
integral is
$\int F_\varepsilon\,d\eta=\sum_{i=1}^K\eta(A_i)h_{z_i}$.
Since the sets $A_i$ are disjoint,
$\sum_{i=1}^K|\eta(A_i)|\le\|\eta\|_{\TV}\le2$, so
$\int F_\varepsilon\,d\eta\in2\operatorname{aco}(H)$. Moreover,
\begin{align*}
\Bigl\|\int h_{u,v}\,d\eta(u,v)-\int F_\varepsilon\,d\eta\Bigr\|
_{L^2(\mathcal W,\lambda)}
&\le\sum_{i=1}^K\int_{A_i}\|h_{z}-h_{z_i}\|_{L^2(\mathcal W,\lambda)}
\,d|\eta|(z)\\
&\le\varepsilon\sum_{i=1}^K|\eta|(A_i)\le2\varepsilon.
\end{align*}
Letting $\varepsilon\to0$ gives $\int h_{u,v}\,d\eta\in
2\overline{\operatorname{aco}}(H)$, and multiplying by
$8\sqrt{\kappa_\mathcal{X}\kappa_\mathcal{Y}}$ yields the inclusion.

\emph{Step 2: Entropy of $\mathcal{A}$.}
By Lemma~\ref{lem:holder-smooth}, there is $L>0$ (depending on $R$, $d_x$
and $d_y$) such that
$\|h_{u,v}-h_{u',v'}\|_{L^2(\mathcal W,\lambda)}
\le L\bigl(\|u-u'\|+\|v-v'\|\bigr)^{1/2}$ for all
$(u,v),(u',v')\in\mathcal X\times\mathcal Y$. Hence any
$(\varepsilon/L)^2$-net of $\mathcal X\times\mathcal Y$ (with respect to
the metric $\|u-u'\|+\|v-v'\|$) induces an $\varepsilon$-net of $H$, and
since $\mathcal X\times\mathcal Y \subset B^{d_x}_R(0) \times B^{d_y}_R(0)$, we have
\[
\mathcal N\bigl(\varepsilon,H,\|\cdot\|_{L^2(\mathcal W,\lambda)}\bigr)
\lesssim_{R,d_x,d_y} \varepsilon^{-2(d_x+d_y)},
\qquad 0<\varepsilon\le1.
\]
Next, $\overline{\operatorname{aco}}(H)$ is a subset of the
closed convex hull
$\overline{\operatorname{conv}}\bigl(H\cup(-H)\cup\{0\}\bigr)$. Since
$\mathcal N\bigl(\varepsilon,H\cup(-H)\cup\{0\},\|\cdot\|_{L^2(\mathcal W,\lambda)}\bigr)\le2\mathcal N\bigl(\varepsilon,H,\|\cdot\|_{L^2(\mathcal W,\lambda)}\bigr)+1$, the entropy bound for convex hulls
\cite[Theorem~2.6.9]{VanDerVaartWellner1996}, applied with
$V=2(d_x+d_y)$, gives
\[
\log \mathcal N\bigl(\varepsilon,\overline{\operatorname{conv}}
\bigl(H\cup(-H)\cup\{0\}\bigr),\|\cdot\|_{L^2(\mathcal W,\lambda)}\bigr)
\lesssim_{R,d_x,d_y} \varepsilon^{-\frac{2V}{V+2}}
=\varepsilon^{-q_{d_x,d_y}},
\qquad 0<\varepsilon\le1.
\]
As $\overline{\operatorname{aco}}(H)\subseteq
\overline{\operatorname{conv}}(H\cup(-H)\cup\{0\})$, the same bound holds
for $\overline{\operatorname{aco}}(H)$, and hence for $\mathcal{A}$ by
Step 1, up to rescaling of the leading constant.

%% file: appendix/semiconcave-failure.tex
\section{Failure of semiconcavity}
\label{app:semiconcavity}
\label{ex:failure-semiconcavity}
Let $\mathcal X=\mathcal Y=[-1/2,1/2]$ and let $\mu=\nu=\operatorname{Unif}([-1/2,1/2])$. In dimension $d=1$, we have $\Sd^0=\{-1,1\}$, $a_1=1$, and $\kappa_{\mathcal X}=\kappa_{\mathcal Y}=1$. Hence $\mathcal W_{\mathcal X}=\mathcal W_{\mathcal Y}=\Sd^0\times[-1/2,1/2]$, with $d\lambda_{\mathcal X}(\theta,t)=d\lambda_{\mathcal Y}(\theta,t)=d\sigma_0(\theta)\,dt$, where $\sigma_0(\{1\})=\sigma_0(\{-1\})=1/2$. Note that $H_{\theta,t}=\{z\in[-1/2,1/2]:\theta z>t\}$. Since $\mu(H_{\theta,t})=1/2-t$ for both $\theta=\pm1$, the centered embedding is
\[
U_x^\mu(\theta,t)=\mathbf 1_{\{\theta x>t\}}-\left(\frac12-t\right)
\]
and similarly for $V_y^\nu$. Take the Dirac measure $\eta=\delta_{(0,0)}$ at $(0,0)\in\mathcal X\times\mathcal Y$. By the defining property of the Bochner integral with respect to a Dirac measure, $a_\eta=8h_{0,0}$, i.e., $a_\eta((\theta,t),(\phi,u))=8\mathbf 1_{\{t<0\}}\mathbf 1_{\{u<0\}}$. We will show that the function $f(x):=c_{a_\eta,\mu,\nu}(x,0)$ with $y=0$ fixed is not semiconcave.\footnote{By symmetry, it would follow that $y\mapsto c_{a_\eta,\mu,\nu}(0,y)$ is not semiconcave either.} Since the two directions $\theta=\pm1$ contribute equally by symmetry, a direct computation gives
\[
\|U_x^\mu\|_{L^2(\mathcal W_{\mathcal X},\lambda_{\mathcal X})}^2
=
\int_{-1/2}^{x}
\left(t+\frac12\right)^2\,dt
+
\int_x^{1/2}
\left(t-\frac12\right)^2\,dt
=
x^2+\frac1{12},
\]
and therefore $\|V_0^\nu\|_{L^2(\mathcal W_{\mathcal Y},\lambda_{\mathcal Y})}^2=1/12$. Now let $L_x:=\int_{\mathcal W_{\mathcal X}}\mathbf 1_{\{t<0\}}U_x^\mu(\theta,t)\,d\lambda_{\mathcal X}(\theta,t)$, so that the integral in the second term in \eqref{eq:cost-original} factorizes as
\[
\int_{\mathcal W}
a_\eta(\xi,\zeta)U_x^\mu(\xi)V_0^\nu(\zeta)\,d\lambda(\xi,\zeta)
=
8L_xL_0.
\]
Writing $I(z):=\int_{-1/2}^{0}\left(\mathbf 1_{\{z>t\}}-\left(\frac12-t\right)\right)\,dt$, we have $I(z)=1/8+z$ for $z\le0$ and $I(z)=1/8$ for $z\ge0$. Since the two directions $\theta=\pm1$ give $I(x)$ and $I(-x)$,
\[
L_x
=
\frac12(I(x)+I(-x))
=
\frac18-\frac{|x|}{2}.
\]
Thus $L_0=1/8$, and consequently $8L_xL_0=L_x=1/8-|x|/2$. Therefore
\[
f(x)
=
-\frac13\left(x^2+\frac1{12}\right)-\frac14+|x|.
\]
For $0<\varepsilon<1/2$,
\[
f(\varepsilon)+f(-\varepsilon)-2f(0)
=
2\varepsilon-\frac23\varepsilon^2.
\]
If $f$ were semiconcave near $0$, then by definition~\cite[Definition~1.1.1]{cannarsa2004semiconcave}, the left-hand side would be bounded above by $C\varepsilon^2$ for all sufficiently small $\varepsilon>0$, which is impossible as $\varepsilon\downarrow0$. Therefore $x\mapsto c_{a_\eta,\mu,\nu}(x,0)$ is not semiconcave.